\documentclass[11pt]{article}
\usepackage[letterpaper, margin=1in]{geometry}

\usepackage{amsmath, amssymb, amsthm, color, mathtools}
\usepackage{caption, graphicx}
\usepackage{subcaption}
\usepackage{enumitem}
\usepackage[longnamesfirst, authoryear]{natbib}
\usepackage{authblk}

\usepackage{changepage}
\usepackage[utf8]{inputenc}
\usepackage{hyperref}
\usepackage{setspace}
\usepackage{xcolor}
\hypersetup{hidelinks,colorlinks=true,linkcolor=blue,citecolor=blue,linktocpage=true}
\usepackage{cleveref}
\newtheorem{theorem}{Theorem}
\newtheorem{corollary}{Corollary}[theorem]
\newtheorem{proposition}{Proposition}
\usepackage[longnamesfirst, authoryear]{natbib}
\usepackage{authblk}

\newtheorem{lemma}[theorem]{Lemma}

\newtheorem{remark}{Remark}

\newcommand{\wc}[1]{[\textcolor{blue}{\textit{WC: #1}}]}

\newcommand{\Eb}{\mathbb{E}}
\newcommand{\Pb}{\mathbb{P}}
\newcommand{\Real}{\mathbb{R}}
\newcommand{\Norm}[1]{\left\Vert#1\right\Vert}
\newcommand{\norm}[1]{\Vert#1\Vert}
\newcommand{\Abs}[1]{\left|#1\right|}
\newcommand{\abs}[1]{|#1|}
\newcommand{\Ip}[1]{\left\langle#1\right\rangle}
\newcommand{\ip}[1]{\langle#1\rangle}
\newcommand{\Set}[1]{\left\{#1\right\}}
\newcommand{\set}[1]{\{#1\}}
\newcommand{\Uv}{\mathbf{U}}
\newcommand{\Yv}{\mathbf{Y}}

\usepackage{comment}

\newcommand{\Ac}{\mathcal{A}}

\newcommand{\Ec}{\mathcal{E}}

\newcommand{\Gc}{\mathcal{G}}

\newcommand{\Nc}{\mathcal{N}}

\newcommand{\Sb}{\mathbb{S}}

\usepackage[title, titletoc]{appendix}
\title{The Berry-Esseen Bound for \\High-dimensional Self-normalized Sums}
\author{Woonyoung Chang$^*$}
\author{Kenta Takatsu$^*$}
\author{Konrad Urban\footnote{Alphabetical order}}
\author{Arun Kumar Kuchibhotla}
\affil{Department of Statistics and Data Science, Carnegie Mellon University}
\date{}

\DeclareMathOperator{\E}{\mathbb{E}}

\usepackage[longnamesfirst, authoryear]{natbib}
\usepackage{algorithm}
\usepackage{algpseudocode}
\begin{document}
\maketitle
\begin{abstract}
This manuscript studies the Gaussian approximation of the coordinate-wise maximum of self-normalized statistics in high-dimensional settings. We derive an explicit Berry-Esseen bound under weak assumptions on the absolute moments. When the third absolute moment is finite, our bound scales as $\log^{5/4}(d)/n^{1/8}$ where $n$ is the sample size and $d$ is the dimension. Hence, our bound tends to zero as long as $\log(d)=o(n^{1/10})$. Our results on self-normalized statistics represent substantial advancements, as such a bound has not been previously available in the high-dimensional central limit theorem (CLT) literature.
\end{abstract}
\section{Introduction}
Let $X_1,\ldots, X_n$ be independent and identically distributed (IID) centered random vectors in $\Real^d$. The vector $T_n\in\Real^d$ of component-wise self-normalized sum of $X_1,\ldots,X_n$ is defined as
\begin{eqnarray*}
    e_j^\top T_n := \frac{|\sum_{i=1}^{n} e_{j}^{\top} X_{i}|}{\sqrt{\sum_{i=1}^{n}(e_{j}^{\top} X_{i})^{2}}}\quad\mbox{for all}\quad j=1,\ldots,d,
\end{eqnarray*} where $e_j$ denotes the $j$:th canonical basis of $\Real^d$.

In the one-dimensional case, the limiting distribution of $T_n$ has been extensively studied. One remarkable property of the self-normalized sum is that it often requires less stringent moment conditions than the
classical limit theorems for the usual sums. A key result by \cite{Gine1997} resolved a conjecture raised by earlier works \citep{Logan1973,Griffin1991,bentkus1996berry} showing that the self-normalized sum $T_n$ is asymptotically normal if and only if $X_1$ belongs to the domain of attraction of the normal law. This was further generalized by \cite{Chistyakov2004}, who provided necessary and sufficient conditions for $T_n$ to converge to an $\alpha$-stable distribution\footnote{We note that the normal distribution is stable with index $\alpha=2$.}. In addition, the distributional approximation of $T_n$ to the standard normal distribution has been explored through various approaches such as uniform Berry-Esseen bounds \citep{Hall1988,bentkus1996berry, Bentkus1996noniid}, non-uniform Berry-Esseen bounds and Cramér-type large deviation \citep{Wang1999, Jing2003, Robinson2005, Hu2009, Beckedorf2025}, and Edgeworth expansion \citep{Finner2010, Derumigny2024, Beckedorf2025}. This list is far from exhaustive.

Extending these results to high-dimensional settings remains surprisingly underexplored despite the growing interest and recent advances in the high-dimensional central limit theorem \citep{kuchibhotla2020high,kuchibhotla2021high,lopes2022central,chernozhuokov2022improved,Chernozhukov2023nearly}. Recent work by \cite{das2024central} establishes near-$n^{-\kappa/2}$ Berry-Esseen rate for $T_n$ in high dimension under the existence of $(2+\kappa)$:th moments of $e_j^\top X_i$ for $0<\kappa\leq 1$. However, their results are restricted to the special case where $e_j^\top X_i$ are independent across both $i$ and $j$. 

This manuscript aims to bridge this gap and studies the Gaussian approximation of $T_n$ under mild conditions and high dimensional settings where the dimension $d$ can grow much faster than the sample size $n$. Specifically, we focus on bounding the Kolmogorov-Smirnov (KS) distance between $\norm{T_n}_\infty$ and $\norm{Z_n}_\infty$, for some $d$-dimensional Gaussian vector $Z_n$, with explicit dependencies of $n$ and $d$ on our bound. To this end, we consider two different Gaussian approximations:
\begin{itemize}
    \item \textbf{Best Gaussian Approximation:}
    \begin{equation}\label{eq:uniform-clt}
        \Delta_n :=\inf_{Z\sim G:G\in\Gc}\sup_{t\in\mathbb{R}}\, \Abs{\mathbb{P}(\|T_{n}\|_\infty \leq t)-\mathbb{P}(\|Z\|_{\infty} \leq t)},
    \end{equation} 
    where $\Gc$ is a collection of mean-zero $d$-dimensional Gaussian distribution whose variance-covariance matrix is a correlation matrix.
    \item \textbf{Moment Matching Gaussian Approximation:}
    \begin{equation}\label{eq:uniform-clt-2}
        \Delta_n^X :=\sup_{t\in\mathbb{R}}\, \Abs{\mathbb{P}(\|T_{n}\|_\infty \leq t)-\mathbb{P}(\|Z^X\|_{\infty} \leq t)},
    \end{equation} 
    where $Z^X$ be centered $d$-dimensional Gaussian with covariance matrix equal to ${\rm Corr}(X_1)$.
\end{itemize}
By construction, $\Delta_n\leq \Delta_n^X$, and $\Delta_n^X$ is well-defined only if $X$ has a finite second moment whereas $\Delta_n$ does not require this. We further show in this manuscript that $\Delta_n$ can converge to $0$ even when the second moment of $X_1$ diverges.

Finally, we discuss how our findings can be utilized for valid inferential procedures in high-dimensional settings. A common approach in the high-dimensional CLT literature is to establish bootstrap consistency, which enables an accurate bootstrap quantile estimation. This can be explored through \emph{moment matching Gaussian approximation} framework in which we are interested in studying $\Delta_n^{*} := \sup_{t}\abs{\mathbb{P}(\|T_{n}^*\|_\infty \leq t|X_1,\ldots,X_n)-\mathbb{P}(\|Z^X\|_{\infty} \leq t)}$ where $T_n^*$ is a bootstrap analog of $T_n$. Alternatively, one can leverage the property that the covariance matrix of the approximating Gaussian distribution is a correlation matrix with unit diagonals. This enables valid yet conservative inferential methods such as Bonferonni or Šidák correction which suggests the quantiles $K_\alpha^B = \Phi^{-1}(1-\alpha/(2p))$ and $K_\alpha^S=\Phi^{-1}((1+(1-\alpha)^{1/p})/2)$ for the nominal probability $1-\alpha$. Formal exploration of these methods is deferred to future work.

\paragraph{Notation.} For two real numbers $a$ and $b$, let $a\vee b=\max\{a,b\}$, $a\wedge b=\min\{a,b\}$, $a_+=\max\{a,0\}$, and $a_-=\max\{-a,0\}$. The $j$:th canonical basis of $\Real^d$ is written as $e_j$ for all $j=1,\ldots,d$, i.e., a $d$-dimensional vector of all zeros except one at the $j$:th element. For any $x\in\Real^d$, we write $\norm{x}_2=\sqrt{x^\top x}$. Moreover, we denote the maximum entry of $x$ as $\norm{x}_\infty=\max_{1\leq j\leq d}\abs{e_j^\top x}$. The unit sphere in $\Real^d$ is $\Sb^{d-1}=\set{\theta\in\Real^d:\norm{\theta}_2=1}$. Let a $k$:th order tensor $T$ be viewed as multi-dimensional arrays in $(\Real^d)^{\otimes k}$. For two $k$:th order tensors $A$ and $B$ of matching dimensions, we define their inner product as $\ip{A,B} :=\sum_{1\leq i_1,\ldots,i_k\leq d}A_{i_1,\ldots,i_k}B_{i_1,\ldots,i_k}$.
\section{Main Results}
Consider independent and identically distributed (IID) mean-zero and non-degenerate variance random vectors $X_{1}, \ldots, X_{n} \in \mathbb{R}^d$ with covariance matrix $\Sigma \in \mathbb{R}^{d\times d}$. We allow the dimension $d$ to grow with $n$ and possibly larger than $n$. We define a self-normalized sum $T_n \in \mathbb{R}^d$ as 
\begin{align}\label{def:Tn}
    e_j^\top T_n := \frac{|\sum_{i=1}^{n} e_{j}^{\top} X_{i}|}{\sqrt{\sum_{i=1}^{n}(e_{j}^{\top} X_{i})^{2}}} \quad\mbox{for all}\quad j=1,\ldots,d.
\end{align}


The main result is stated in terms of the truncated moments, which permits the random vector $X_1$ to have a diverging second moment. Given a random vector $X_1$ and $n \ge 1$, we define a non-negative number $a_j$, which satisfies the following inequality: 
\begin{align}
    a_{j}^{2}:=\sup \{b \ge 0: \E[(e_{j}^{\top} X_{1})^{2} \mathbf{1}\{(e_{j}^{\top} X_{1})^{2} \leq b n\}] \geq b\}.
\end{align}
For $1 \le i \le n$ and $1 \le j \le d$, we define the truncated random vector
\begin{align}\label{def:truncted_rv}
    e_{j}^{\top} Y_{i}:=\frac{e_{j}^{\top} X_{i}}{a_{j} n^{1 / 2}} \mathbf{1}\{(e_{j}^{\top} X_{i})^{2} \leq a_{j}^{2} n\}.
\end{align}
By Lemma 1.3 of \citet{bentkus1996berry}, we have $\|Y_i\|_\infty \le 1$ almost surely and $\E[(e_j^\top Y_i)^2] = 1/n$ for all $1 \le j \le d$. We provide the main results below:
\begin{theorem}\label{thm:berry-esseen}  There exists an absolute constant $C > 0$ such that 
\begin{align*}
    \Delta_n &\le C\Bigg(n \mathbb{P}\left(\max _{1 \leq j \leq d} \frac{(e_{j}^{\top} X_{1})^{2}}{a_{j}^{2}}>n\right)  + \left(n\log^{5/2}(ed)\norm{\Eb[Y_1]}_\infty\right)^{1/2}+\left(n\log^5(ed)\Eb\norm{Y_1}_\infty^3\right)^{1/4}\Bigg).
\end{align*}
Additionally, there exists an absolute constant $C > 0$ such that 
\begin{align*}
    \Delta_n^X &\le C\Bigg(n \mathbb{P}\left(\max _{1 \leq j \leq d} \frac{(e_{j}^{\top} X_{1})^{2}}{a_{j}^{2}}>n\right) + \log(ed)\max_{1\le j \le d}\Set{\mathbb{E}\left[\frac{(e_{j}^{\top} X_{1})^2}{ a_{j}^2} \mathbf{1}\{(e_{j}^{\top} X_{1})^{2} > a_{j}^{2} n\}\right]}^{1/2}\\
    &\qquad + \left(n\log^{5/2}(ed)\norm{\Eb[Y_1]}_\infty\right)^{1/2}+\left(n\log^5(ed)\Eb\norm{Y_1}_\infty^3\right)^{1/4}\Bigg).
\end{align*}
\end{theorem}
\begin{corollary}\label{cor:berry-esseen}
Denote $\E (e_j^\top X_1)^2 = \sigma_{j}^2$ and suppose that $\E[\max_{1\le j \le d}|e_j^\top X_1/\sigma_{j}|^{2+\delta}] < \infty$ for some $\delta \in (0, 1]$, then the bound in Theorem~\ref{thm:berry-esseen} reduces to
    \begin{align*}
    \Delta_n \le  C\log^{5/4}(ed)n^{-\delta/8}\left(\E \max_{1\le j \le d}|e_j^\top X_1/\sigma_{j}|^{2+\delta}\right)^{1/4}.
\end{align*}
In particular, when $\E[\max_{1\le j \le d}|e_j^\top X_1/\sigma_{j}|^{3}] < \infty$, one has
\begin{align}\label{eq:be-bound-simplified}
    \Delta_n \le  C \log^{5/4}(ed)n^{-1/8}\left(\E \max_{1\le j \le d}|e_j^\top X_1/\sigma_{j}|^3\right)^{1/4}.
\end{align}
\end{corollary}

\begin{remark}[On the rate of convergence]\label{remark:improvements}
    The result \eqref{eq:be-bound-simplified} suggests that $\Delta_n$ tends to zero as long as $\log d = o(n^{1/10})$. This manuscript also establishes that the rate of convergence of $\Delta_n$ is of the order $(\log d)^{5/4}n^{-1/8}$ under the most favorable case with the finite third moment of $\|X_1\|_\infty$. Without considering self-normalization, the approximation error has been shown to converge as fast as $\mathrm{polylog}(dn)n^{-1/2}$ in \citet{kuchibhotla2020high}. Corresponding improvements under self-normalization require significant efforts due to the complex dependence between the summands of the self-normalized sum. To the best of our knowledge, the results corresponding to \Cref{thm:berry-esseen} are not available in the literature without strong assumptions; \citet{das2024central}, for instance, provides the bound on $\Delta_n$ under the assumption that the components of $X_1$ are also IID. 
\end{remark}
\subsection{Sketch of the proof}
This section outlines the proof of Theorem~\ref{thm:berry-esseen}. The proof follows similarly to that of Theorem 1.4 of \cite{bentkus1996berry}, with crucial intermediate steps differing. We first introduce some notation. For each truncated random variable $Y_i$, defined in \eqref{def:truncted_rv}, the corresponding self-normalized sum for $1 \le j \le d$ is denoted by:
\begin{align}\label{def:TnY}
    e_j^\top T_n^Y := \frac{|\sum_{i=1}^{n} e_{j}^{\top} Y_{i}|}{\sqrt{\sum_{i=1}^{n}(e_{j}^{\top} Y_{i})^{2}}}.
\end{align}
Let $Z := \sum_{i=1}^n Z_i$ be the sum of $n$ IID mean-zero Gaussian random vectors such that ${\rm Var}(Y_i)={\rm Var}(Z_i)$ for all $i=1,\ldots,n$. 
Let $g: \mathbb{R} \mapsto \mathbb{R}$ be any infinitely differentiable with bounded derivatives, satisfying 
\begin{align}
    \frac{1}{8} \leq g(x) \leq 2 \quad \mbox{for all} \quad x \in \mathbb{R}, \quad \mbox{and}\quad g(x)=\frac{1}{\sqrt{|x|}}, \quad \mbox{if}\quad |x| \in\left[\frac{1}{4}, \frac{7}{4}\right]. 
\end{align}
We define a random vector $\widetilde{Y}_i\in\mathbb{R}^d$ via 
\begin{align}\label{def:Ytilde}
    e_j^\top\widetilde{Y}_i := e_j^\top Y_{i}g(1+\eta_{j, n})  \quad \mbox{with}\quad \eta_{j, n}=\sum_{i=1}^{n}\left\{(e_{j}^{\top} Y_{i})^{2}-\frac{1}{n}\right\}.
\end{align}
Finally, the sums of random vectors are denoted by $\Yv := \sum_{i=1}^nY_i$ and $\widetilde{\Yv}:=\sum_{i=1}^n\widetilde{Y}_i$. With these notations in place, we proceed with the main proof. First, by the triangle inequalities, we obtain 
\begin{align*}
    \Delta_n &\le \sup_{t\in\mathbb{R}}\, |\mathbb{P}(\|T_{n}\|_\infty \leq t)-\mathbb{P}(\|T_{n}^Y\|_{\infty} \leq t)| + \sup_{t\in\mathbb{R}}\, |\mathbb{P}(\|T_{n}^Y\|_\infty \leq t)-\mathbb{P}(\|\widetilde{\Yv}\|_{\infty} \leq t)|\\
    &\qquad +\sup_{t\in\mathbb{R}}\, |\mathbb{P}(\|\widetilde{\Yv}\|_\infty \leq t)-\mathbb{P}(\|Z\|_{\infty} \leq t)| \\
    &= \mathbf{I_1} + \mathbf{I_2} + \sup_{t\in\mathbb{R}}\, |\mathbb{P}(\|\widetilde{\Yv}\|_\infty \leq t)-\mathbb{P}(\|Z\|_{\infty} \leq t)|
\end{align*}
and 
\begin{align*}
    \Delta_n^X &\le \mathbf{I_1} + \mathbf{I_2} + \sup_{t\in\mathbb{R}}\, |\mathbb{P}(\|\widetilde{\Yv}\|_\infty \leq t)-\mathbb{P}(\|Z^X\|_{\infty} \leq t)|+\sup_{t\in\mathbb{R}}\, |\mathbb{P}(\|Z\|_\infty \leq t)-\mathbb{P}(\|Z^X\|_{\infty} \leq t)|. 
\end{align*}
\Cref{lemma:remainder1} and \cref{lemma:remainder2} establish with some universal constant $C>0$,
\begin{align*}
    \mathbf{I_1} + \mathbf{I_2} \le n \mathbb{P}\left(\max _{1 \leq j \leq d} \frac{(e_{j}^{\top} X_{1})^{2}}{a_{j}^{2}}>n\right) + C\log(ed)n \mathbb{E}\|Y_1\|_\infty^3.
\end{align*}
\Cref{lemma:remainder3} also proves with some universal constant $C > 1$,  
\begin{align*}
    \sup_{t\in\mathbb{R}}\, |\mathbb{P}(\|Z\|_\infty \leq t)-\mathbb{P}(\|Z^X\|_{\infty} \leq t)| \le C\log(ed) R_n^{1/2},
\end{align*}
where 
    \begin{align*}
        R_n := \mathbb{E}\left[\frac{(e_{j}^{\top} X_{1})^2}{ a_{j}^2} \mathbf{1}\{(e_{j}^{\top} X_{1})^{2} > a_{j}^{2} n\}\right] + n^2\|\mathbb{E}[Y_1]\|_\infty^2.
    \end{align*}
    
It remains to control $\sup_{t\in\mathbb{R}}\, |\mathbb{P}(\|\widetilde{\Yv}\|_\infty \leq t)-\mathbb{P}(\|Z\|_{\infty} \leq t)|$. For $\varepsilon > 0$, we define a smooth approximation of the indicator function $H_{\varepsilon,t} : \mathbb{R}^d \mapsto \mathbb{R}$ by Gaussian convolution as 
\begin{align}\label{def:H-smooth}
    H_{\varepsilon,t}(x) = \mathbb{E}[\mathbf{1}\{\|x+\varepsilon W\|_\infty \le t\}]\quad\mbox{where}\quad W \overset{d}{=} N(0, I_{d}).
\end{align}
Then, by the smoothing lemma (such as Lemma 2.4 of \citet{fang2021high} and Lemma 1 of \citet{kuchibhotla2020high}), 
\begin{align*}
    &\sup_{t\in\mathbb{R}}\, |\mathbb{P}(\|\widetilde{\Yv}\|_\infty \leq t)-\mathbb{P}(\|Z\|_{\infty} \leq t)| \\
    &\qquad \le \sup_{t\in\mathbb{R}}\, |\mathbb{P}(\|\widetilde{\Yv}\|_\infty \leq t)-\mathbb{P}(\|\Yv\|_{\infty} \leq t)| + \sup_{t\in\mathbb{R}}\, |\mathbb{P}(\|\Yv\|_\infty \leq t)-\mathbb{P}(\|Z\|_{\infty} \leq t)|\\
    &\qquad \le \sup_{t\in\mathbb{R}}\, |\mathbb{E}[H_{\varepsilon,t}(\widetilde{\Yv})]-\mathbb{E}[H_{\varepsilon,t}(\Yv)]| + \sup_{t\in\mathbb{R}}\, |\mathbb{P}(\|\Yv\|_\infty \leq t)-\mathbb{P}(\|Z\|_{\infty} \leq t)|+ C\varepsilon \log(ed) ,
\end{align*} for any $\varepsilon>0$ and a universal constant $C>0$.
We prove that for a possibly different universal constant $C>0$,
\begin{align}\label{eq:pf.thm1.1}
    &\inf_{\varepsilon>0}\Set{ \sup_{t\in\mathbb{R}}\, |\mathbb{E}[H_{\varepsilon,t}(\widetilde{\Yv})]-\mathbb{E}[H_{\varepsilon,t}(\Yv)]| + C\varepsilon \log(ed)}\nonumber\\
    &\qquad \leq C\Set{\left(n\log^{5/2}(ed)\norm{\Eb[Y_1]}_\infty\right)^{1/2}+\left(n\log^5(ed)\Eb[\norm{Y_1}_\infty^3]\right)^{1/4}},
\end{align} as long as the right-hand side is no greater than $1$. We further prove that
\begin{equation}\label{eq:pf.thm1.2}
    \sup_{t\in\mathbb{R}}\, |\mathbb{P}(\|\Yv\|_\infty \leq t)-\mathbb{P}(\|Z\|_{\infty} \leq t)|\leq C\Set{n\log^{1/2}(ed)\norm{\Eb[Y_1]}_\infty+\left(n\log^5(ed)\Eb\norm{Y_1}_\infty^3\right)^{1/4}},
\end{equation} provided that the right-hand side is no greater than 1. Putting all together implies Theorem~\ref{thm:berry-esseen}. 
\section{Comments on the proof technique and refinements}
As mentioned in \Cref{remark:improvements}, the Berry-Esseen bound for the sum of independent random variables can be improved up to $\mathrm{polylog}(dn)n^{-1/2}$ in high-dimensional settings \citep{kuchibhotla2020high}. While the analogous refinements may be feasible by following the general arguments in \citet{kuchibhotla2021high} or \citet{kuchibhotla2020high}, the derivation will be considerably different due to the complex dependence among the self-normalized sum.

The current bottleneck is precisely in the handling of the following triangle inequality:
\begin{align*}
    &\sup_{t\in\mathbb{R}}\, |\mathbb{P}(\|\widetilde{\Yv}\|_\infty \leq t)-\mathbb{P}(\|Z\|_{\infty} \leq t)| \\
    &\qquad \le \sup_{t\in\mathbb{R}}\, |\mathbb{P}(\|\widetilde{\Yv}\|_\infty \leq t)-\mathbb{P}(\|\Yv\|_{\infty} \leq t)| + \sup_{t\in\mathbb{R}}\, |\mathbb{P}(\|\Yv\|_\infty \leq t)-\mathbb{P}(\|Z\|_{\infty} \leq t)|
\end{align*}
where we directly compared the distributions of $\|\widetilde{\Yv}\|_\infty$ and $\|\Yv\|_\infty$. This is an approach took by \citet{bentkus1996berry, Bentkus1996noniid} for the univariate settings as well. We may consider an alternative approach for the refinement. First, we define the following mapping of any random vector $\xi \in \mathbb{R}^d$, 
\begin{align*}
    \varphi_j(\xi_1, \ldots, \xi_n) := \sum_{i=1}^n \xi_{i}g(1+\eta^\xi_{j, n})  \quad \mbox{with} \quad\eta^\xi_{j, n}=\sum_{i=1}^{n}\left\{(e_{j}^{\top} \xi_{i})^{2}-1/n\right\}.
\end{align*}
Adopting this notation to $(Y_1, \ldots, Y_n)$ and $(Z_1, \ldots, Z_n)$, we can define
\begin{align*}
    e_j^\top \widetilde{\Yv} := \varphi_j(Y_1, \ldots, Y_n) \quad \mbox{and}\quad  e_j^\top \widetilde{\mathbf{Z}} := \varphi_j(Z_1, \ldots, Z_n).
\end{align*}
We can now control the earlier inequality as follows:
\begin{align*}
    &\sup_{t\in\mathbb{R}}\, |\mathbb{P}(\|\widetilde{\Yv}\|_\infty \leq t)-\mathbb{P}(\|Z\|_{\infty} \leq t)| \\
    &\qquad \le \sup_{t\in\mathbb{R}}\, |\mathbb{P}(\|\widetilde{\Yv}\|_\infty \leq t)-\mathbb{P}(\|\widetilde{\mathbf{Z}}\|_{\infty} \leq t)| + \sup_{t\in\mathbb{R}}\, |\mathbb{P}(\|\widetilde{\mathbf{Z}}\|_\infty \leq t)-\mathbb{P}(\|Z\|_{\infty} \leq t)|
\end{align*}
where the second term can be easily managed by the argument presented in this manuscript. We are now left with the first term. For $k \in \{0,\ldots, n\}$, we define 
\begin{align*}
    (Y_{1:k}, Z_{(k+1):n}) &:= (Y_1, \ldots, Y_k, Z_{k+1}, \ldots, Z_n).
\end{align*}
In particular, $(Y_{1:k}, Z_{(k+1):n}) \equiv (Y_1, \ldots,  Y_n)$ when $k=n$ and  $(Y_{1:k}, Z_{(k+1):n}) \equiv (Z_1, \ldots,  Z_n)$ when $k=0$.
After invoking the smoothing inequality with telescoping, the expression reduces to 
\begin{align*}
    &\sup_{t\in\mathbb{R}}\, |\mathbb{P}(\|\widetilde{\Yv}\|_\infty \leq t)-\mathbb{P}(\|\widetilde{\mathbf{Z}}\|_{\infty} \leq t)| \\
    &\qquad \le\sup_{t\in\mathbb{R}}\, |\mathbb{E}[H_{\varepsilon,t}\circ\varphi(Y_1, \ldots, Y_n) ]-\mathbb{E}[H_{\varepsilon,t}\circ\varphi(Z_1, \ldots, Z_n)]| + C\varepsilon \log(ed) \\
    &\qquad \le \sum_{j=0}^n \sup_{t\in\mathbb{R}}\, |\mathbb{E}[H_{\varepsilon,t}\circ \varphi(Y_{1:(n-j-1)}, Z_{(n-j):n})]-\mathbb{E}[H_{\varepsilon,t}\circ \varphi(Y_{1:(n-j)}, Z_{(n-j+1):n})]| + C\varepsilon \log(ed).
\end{align*}
This procedure performs the classical Lindeberg swapping through the nonlinear map $H_{\varepsilon,t}\circ \varphi$. Additionally, as the derivatives of $H_{\varepsilon,t}$ are zero over most regions, sharper bounds on the remainder term from the Taylor series expansion may be expected. This approach has been employed for the high-dimensional CLT without self-normalization---see Section 6 and Step C onward in \citet{kuchibhotla2021high}. We anticipate that this approach will improve the $n^{-1/8}$ term to $n^{-1/6}$ in \eqref{eq:be-bound-simplified}.
\section{Proof of Theorem~\ref{thm:berry-esseen}}
We may assume that $n\log^{5/2}(ed)\norm{\Eb[Y_1]}_\infty\leq 1$ and $n\log^5(ed)\Eb\norm{Y_1}_\infty^3\leq 1$ since otherwise Theorem~\ref{thm:berry-esseen} holds by simply taking a sufficiently large universal constant. 

The rest of the section is dedicated to the proofs of \eqref{eq:pf.thm1.1} and \eqref{eq:pf.thm1.2}.

\paragraph{Proof of \eqref{eq:pf.thm1.1}}
Lemma~\ref{lemma:Y_Ytilde} provides an upper bound for $|\mathbb{E}[H_{\varepsilon,t}(\widetilde{\Yv})]-\mathbb{E}[H_{\varepsilon,t}(\Yv)]|$ where $h_j=h_j(\varepsilon)$ (defined as \eqref{def:hj} in Lemma~\ref{lemma:Y_Ytilde} below), for $j=1,\ldots,4$, can be further bounded via Lemma~2.3 of \cite{fang2021high} as
\begin{equation*}
    h_j\le C\varepsilon^{-j}(\log d)^{j/2},
\end{equation*} for a constant $C$ only depends on $j\in\mathbb{N}$. For the sake of simplicity, we denote $\mu_3 = n\Eb[\norm{Y_1}_\infty^3]$ and $\mu_1 = n\norm{\Eb[Y_1]}_\infty$. We then aim to minimize
\begin{align*}
    \phi(\varepsilon)&=\varepsilon^{-1}\left(\log^{3/2}(ed)\mu_3+\log^{1/2}(ed)\mu_1\right) + \varepsilon^{-2}\left(\log^{3}(ed)\mu_3+\log(ed)\mu_1^2\right) \\
    &\qquad+\varepsilon^{-3}\log^{3/2}(ed)\mu_3+\varepsilon^{-4}\log^{2}(ed)\mu_3^2+\varepsilon\log(ed).
\end{align*} We set
\begin{align*}
    \varepsilon &= \log(ed)^{1/4}(\mu_3^{1/2}+\mu_1^{1/2})+\log(ed)^{2/3}\mu_3^{1/3}+\mu_1^{2/3}+\log(ed)^{1/8}\mu_3^{1/4}+\log(ed)^{1/5}\mu_3^{2/5}.
\end{align*} This choice allows $\phi(\varepsilon)\leq 4\varepsilon\log(ed)$, and we further deduce that
\begin{align*}
    \frac{1}{4}\phi(\varepsilon)&\leq \left(\log^{5/2}(ed)\mu_3\right)^{1/2}+\left(\log^{5}(ed)\mu_3\right)^{1/3}+\left(\log^{9/2}(ed)\mu_3\right)^{1/4}\\
    &\qquad+\left(\log^{3}(ed)\mu_3\right)^{2/5}+\left(\log^{5/2}(ed)\mu_1\right)^{1/2}+\left(\log^{3/2}(ed)\mu_1\right)^{2/3}\\
    &\leq4\left(\log^5(ed)\mu_3\right)^{1/4}+2\left(\log^{5/2}(ed)\mu_1\right)^{1/2}.
\end{align*} This proves \eqref{eq:pf.thm1.1}.
\paragraph{Proof of \eqref{eq:pf.thm1.2}}
From a simple implication, it first follows that 
\begin{align*}
    &\sup_{t\in\mathbb{R}}\, \Set{\mathbb{P}(\|\Yv\|_\infty \leq t)-\mathbb{P}(\|Z\|_{\infty} \leq t)}\\
    &\qquad \le \sup_{t\in\mathbb{R}}\,\Set{\mathbb{P}(\|\Yv-\E[\Yv]\|_\infty -\|\E[\Yv]\|_\infty \leq t)-\mathbb{P}(\|Z\|_{\infty} \leq t)}
    \\
    &\qquad \le \sup_{t\in\mathbb{R}}\,\bigg\{|\mathbb{P}(\|\Yv-\E[\Yv]\|_\infty \leq t + \|\E[\Yv]\|_\infty)-\mathbb{P}(\|Z\|_{\infty} \leq t+\|\E[\Yv]\|_\infty)\\
    &\qquad\qquad+\mathbb{P}(\|Z\|_{\infty} \leq t+\|\E[\Yv]\|_\infty)-\mathbb{P}(\|Z\|_{\infty} \leq t)\bigg\} \\
    &\qquad \le \sup_{t\in\mathbb{R}}\,|\mathbb{P}(\|\Yv-\E[\Yv]\|_\infty \leq t)-\mathbb{P}(\|Z\|_{\infty} \leq t)| + \sup_{t\in\mathbb{R}}\mathbb{P}(t <\|Z\|_{\infty} \leq t+\|\E[\Yv]\|_\infty).
\end{align*} Similarly, one can show that
\begin{align*}
    &\sup_{t\in\mathbb{R}}\, \Set{\mathbb{P}(\|\Yv\|_\infty >t)-\mathbb{P}(\|Z\|_{\infty} > t)}\\
    &\qquad \le \sup_{t\in\mathbb{R}}\,|\mathbb{P}(\|\Yv-\E[\Yv]\|_\infty > t)-\mathbb{P}(\|Z\|_{\infty} > t)| + \sup_{t\in\mathbb{R}}\mathbb{P}(t-\|\E[\Yv]\|_\infty<\|Z\|_{\infty} \leq t).
\end{align*} Hence, we get
\begin{align*}
    &\sup_{t\in\mathbb{R}}\, \Abs{\mathbb{P}(\|\Yv\|_\infty >t)-\mathbb{P}(\|Z\|_{\infty} > t)}\le \sup_{t\in\mathbb{R}}\,|\mathbb{P}(\|\Yv-\E[\Yv]\|_\infty \leq t)-\mathbb{P}(\|Z\|_{\infty} \leq t)|\\
    &\qquad+~\sup_{t\in\mathbb{R}}\mathbb{P}(t-\|\E[\Yv]\|_\infty<\|Z\|_{\infty} \leq t + \|\E[\Yv]\|_\infty).
\end{align*}
The last term in the display can be controlled by invoking Gaussian anti-concentration inequality, also known as Nazarov's inequality \citep{nazarov2003maximal}. In particular, Theorem 1 of \citet{chernozhukov2017detailed} applies and yields
\begin{align*}
    \sup_{t\in\mathbb{R}}\mathbb{P}(t-\|\E[\Yv]\|_\infty<\|Z\|_{\infty} \leq t + \|\E[\Yv]\|_\infty)\leq C\sqrt{\log(ed)}\Norm{\Eb[\Yv]}_\infty = Cn\sqrt{\log(ed)}\Norm{\Eb[Y_1]}_\infty,
\end{align*} where $C$ represents a universal constant\footnote{An explicit constant can be found in \cite{chernozhukov2017detailed}}. To control the remaining term
\begin{align*}
    \vartheta_n:=\Abs{\mathbb{P}(\|\Yv-\E[\Yv]\|_\infty \leq t)-\mathbb{P}(\|Z\|_{\infty} \leq t)},
\end{align*} we recall that $\Yv-\E[\Yv]$ and $Z$ are the sums of centered independent random vectors with matching second moments, and thus, the high-dimensional CLT results are applicable. In particular, we use Proposition~\ref{prop:B.1}, which refines Theorem~2.5 of \cite{chernozhuokov2022improved}. Setting $b_1=b_2=1$, $q=3$, $B_n = n^{1/2}\max_{1\leq j\leq d}(\Eb[(e_j^\top Y_1)^4])^{1/4}\leq n^{1/2}(\Eb\norm{Y_1}_\infty^3)^{1/4}$, and $D_n = n^{1/2}(\Eb\norm{Y_1}_\infty^3)^{1/3}$, it leads to
\begin{align*}
    \vartheta_n&\leq C\left[\left(n\log^5(ed)\Eb\norm{Y_1}_\infty^3\right)^{1/4}+\left(n\log^{9/2}(ed)\Eb\norm{Y_1}_\infty^3\right)^{1/3}\right]\\
    &\leq C\left(n\log^5(ed)\Eb\norm{Y_1}_\infty^3\right)^{1/4},
\end{align*} for some universal constant $C>0$ where the last inequality holds as long as $n\log^{5}(ed)\Eb\norm{Y_1}_\infty^3\leq 1$. This proves \eqref{eq:pf.thm1.2}.

\begin{proof}[\bfseries{Proof of \Cref{cor:berry-esseen}}]
First, observe that $\sigma_j^2 = \E[(e_{j}^{\top} X_{1})^2]$ and 
    \begin{align*}
        1 -\mathbb{E}\left[\left|\frac{e_{j}^{\top} X_1}{\sigma_{j} n^{1 / 2}}\right|^2 \mathbf{1}\{(e_{j}^{\top} X_1)^{2} \le a_{j}^{2} n\}\right] = \mathbb{E}\left[\left|\frac{e_{j}^{\top} X_1}{\sigma_{j} n^{1 / 2}}\right|^2 \mathbf{1}\{(e_{j}^{\top} X_1)^{2} > a_{j}^{2} n\}\right] .
    \end{align*}
    When $a_j^2 \le \sigma_j^2/2$, it follows that 
    \begin{align*}
        \mathbb{E}\left[\left|\frac{e_{j}^{\top} X_1}{\sigma_{j} n^{1 / 2}}\right|^2 \mathbf{1}\{(e_{j}^{\top} X_1)^{2} > \sigma_{j}^{2} n/2\}\right] \ge 1/2.
    \end{align*}
    Hence, there exists a constant $C \ge 2$ such that 
    \begin{align*}
    \Delta_n \le C\mathbb{E}\left[\left|\frac{e_{j}^{\top} X_1}{\sigma_{j} n^{1 / 2}}\right|^2 \mathbf{1}\{(e_{j}^{\top} X_1)^{2} > \sigma_{j}^{2} n/2\}\right] \le C\mathbb{E}\left[\frac{(e_{j}^{\top} X_{1}/\sigma_j)^{2+\delta}}{ n^{\delta/2}} \right]
    \end{align*}
    since $\Delta_n \le 1$ trivially by definition.
    Thus we may focus on the case $\sigma_j^2/2 \le a_j^2 \le \sigma_j^2$ where the second inequality follows from \Cref{lemma:truncation}. To apply Theorem~\ref{thm:berry-esseen}, we make the following observations: first by Markov inequality, we have 
\begin{align*}
    n \mathbb{P}\left(\max _{1 \leq j \leq d} \frac{(e_{j}^{\top} X_1)^{2}}{a_{j}^{2}}> n\right) \le n \frac{\E \max_{1\le j \le d}\,|e_j^\top X_1/a_j|^{2+\delta}}{n^{1+\delta/2}} \le Cn^{-\delta/2}\E\left[\max_{1\le j \le d}|e_j^\top X_1/\sigma_{j}|^{2+\delta}\right].
\end{align*}
Since $X_1$ is centered, we have
\begin{align*}
    n\|\E Y_1\|_\infty &= n\max_{1\le j \le d}\left|\mathbb{E}\left[\frac{e_{j}^{\top} X_1}{a_j n^{1 / 2}} -\frac{e_{j}^{\top} X_1}{a_j n^{1 / 2}} \mathbf{1}\{(e_{j}^{\top} X_1)^{2} \leq a_j^{2} n\}\right]\right|\\
    &= n\max_{1\le j \le d}\left|\mathbb{E}\left[\frac{e_{j}^{\top} X_1}{a_j n^{1 / 2}} \mathbf{1}\{(e_{j}^{\top} X_1)^{2} > a_j^{2} n\}\right]\right|\le Cn^{-\delta/2}\E\left[\max_{1\le j \le d}|e_j^\top X_1/\sigma_{j}|^{2+\delta}\right].
\end{align*}
Next, we obtain 
\begin{align*}
    n\E \|Y_1\|^3_\infty &= n \E \left[\max_{1\le j \le d}\, \frac{|e_j^\top X_{1}|^3}{a_j^3 n^{3 / 2}} \mathbf{1}\{(e_{j}^{\top} X_1)^{2} \leq a_j^{2} n\}\right] \le  Cn^{-\delta/2}\E\left[\max_{1\le j \le d}|e_j^\top X_1/\sigma_{j}|^{2+\delta}\right].
\end{align*}
Finally, we obtain 
\begin{align*}
     \max_{1\le j \le d}\mathbb{E}\left[\frac{(e_{j}^{\top} X_{1})^2}{ a_{j}^2} \mathbf{1}\{(e_{j}^{\top} X_{1})^{2} > a_{j}^{2} n\}\right] \le C\max_{1\leq j\leq d}\mathbb{E}\left[\frac{(e_{j}^{\top} X_{1})^{2+\delta}}{ n^{\delta/2}\sigma_j^{2+\delta}} \right].
\end{align*}
Thus, the claim follows.
\end{proof}
\subsection{Technical Lemmas}
\begin{lemma}\label{lemma:truncation}
    For a centered random variable $X$ with $\Eb[X^2]=\sigma^2$, let
    \begin{equation}
        a = \sup\Set{b\geq 0: \Eb\left[X^2\mathbf{1}(X^2\leq b^2n)\right]\geq b^2}.
    \end{equation} Then, $a\leq\sigma$ and $a$ is the largest solution of the equation of
    \begin{eqnarray*}
        \Eb\left[X^2\mathbf{1}(X^2\leq a^2n)\right] = a^2.
    \end{eqnarray*}
\end{lemma}
\begin{proof}[\bfseries{Proof of \Cref{lemma:truncation}}]
    See Lemma~1.3 of \cite{bentkus1996berry}.
\end{proof}

\begin{lemma}\label{lemma:remainder1}
Let $T_n, T_n^Y \in \mathbb{R}^d$ be defined as \eqref{def:Tn} and \eqref{def:TnY}. Then 
    \begin{align*}
        \sup_{t\in\mathbb{R}}\, |\mathbb{P}(\|T_{n}\|_\infty \leq t)-\mathbb{P}(\|T_{n}^Y\|_{\infty} \leq t)| \le n \mathbb{P}\left(\max _{1 \leq j \leq d} \frac{(e_{j}^{\top} X_{1})^{2}}{a_{j}^{2}}>n\right).
    \end{align*}
\end{lemma}
\begin{proof}[\bfseries{Proof of \Cref{lemma:remainder1}}] We denote
\begin{eqnarray*}
    \delta_1 := \sup_{t\in\Real}\Abs{\Pb(\|T_{n}\|_\infty \leq t)-\mathbb{P}(\|T_{n}^Y\|_{\infty} \leq t)}.
\end{eqnarray*} Furthermore, we define
\begin{eqnarray*}
    \Ec := \Set{e_j^\top X_i = e_j^\top Y_i \mbox{ for all }1\leq i\leq n\mbox{ and }1\leq j\leq p.}.
\end{eqnarray*} It follows from the definition that $T_n = T_n^Y$ on $\Ec$. Hence,
\begin{eqnarray*}
    \delta_1\leq \sup_{t\in\Real}\Abs{\Pb(\set{\|T_{n}\|_\infty \leq t}\cap\Ec^\complement)-\mathbb{P}(\set{\|T_{n}^Y\|_{\infty} \leq t}\cap\Ec^\complement)}\leq \Pb(\Ec^\complement).
\end{eqnarray*} With a union bound, it is immediate that 
\begin{align*}
    \mathbb{P}(\mathcal{E}^\complement) = \mathbb{P}\left(\exists\, i \in \{1, \ldots, n\}\, : \, \max _{1 \leq j \leq d} \frac{(e_{j}^{\top} X_i)^{2}}{a_{j}^{2}}> n\right) \le \sum_{i=1}^n \mathbb{P}\left(\max _{1 \leq j \leq d} \frac{(e_{j}^{\top} X_i)^{2}}{a_{j}^{2}}> n\right).
\end{align*}
Since $X_i$ is identically distributed, we conclude the claim.
\end{proof}
\begin{lemma}\label{lemma:remainder2}
    Let $T_n^Y \in \mathbb{R}^d$ be defined as \eqref{def:TnY} and $\widetilde{\Yv} = \sum_{i=1}^n \widetilde{Y}_i$ where $\widetilde{Y}_i$ is defined as \eqref{def:Ytilde}. Then there exists a universal constant $C > 0$ such that 
    \begin{align*}
        \sup_{t\in\mathbb{R}}\, |\mathbb{P}(\|T_{n}^Y\|_\infty \leq t)-\mathbb{P}(\|\widetilde{\Yv}\|_{\infty} \leq t)|\le C\log(ed)n \mathbb{E}\|Y_1\|_\infty^3.
    \end{align*}
\end{lemma}
\begin{proof}[\bfseries{Proof of \Cref{lemma:remainder2}}]
By the definition of the function $g$, we have $T_{n}^Y=\widetilde{\Yv}$ almost surely on the event where 
\begin{align*}
    \left\{|1+\eta_{j,n}| \in \left[\frac{1}{4}, \frac{7}{4}\right] \mbox{ for all $1 \le j \le d$}\right\}.
\end{align*}
It thus suffices to control the complement of this event as
\begin{align*}
    \sup _{t \in \mathbb{R}}|\mathbb{P}(T_{n}^Y \leq t)-\mathbb{P}(\widetilde{\Yv}\leq t)| &\le \mathbb{P}\left( \exists j \in \{1, 2, \ldots, d\} \, : \, |1+\eta_{j,n}| \notin \left[\frac{1}{4}, \frac{7}{4}\right] \right) \\
    &\le \mathbb{P}\left( \max_{1\le j \le d}\, |\eta_{j,n}| > \frac{1}{2}\right) \\
    &= \mathbb{P}\left( \max_{1\le j \le d}\, |\sum_{i=1}^{n}(e_{j}^{\top} Y_{i})^{2}-1| > \frac{1}{2}\right) \\
    &\le 4\mathbb{E}\left[ \max_{1\le j \le d}\, |\sum_{i=1}^{n}(e_{j}^{\top} Y_{i})^{2}-1|^2 \right],
\end{align*} where the last inequality is Chebyshev inequality. The last term is controlled by \Cref{lemma:moment-bounds-for-Ustats} as:
\begin{align*}
    \mathbb{E}\left[ \max_{1\le j \le d}\, |\sum_{i=1}^{n}(e_{j}^{\top} Y_{i})^{2}-1|^2 \right] \lesssim n\log(ed)\left(\Eb[\norm{Y_1}_\infty^3]\right).
\end{align*}
Thus the claim is concluded. 
\end{proof}
\begin{lemma}\label{lemma:remainder3}
      Let $Z =\sum_{i=1}^n Z_i$ and $Z^X=\sum_{i=1}^n Z^X_i$ be centered Gaussian random vectors in $\mathbb{R}^d$ such that ${\rm Var}(Y_i)={\rm Var}(Z_i)$ and $\mathbb{E}[Z^X Z^{X\top}] = \Omega^X$ where $\Omega^X$ is a correlation matrix of $X$. Define
      \begin{eqnarray*}
          R_n := \max_{1\le j \le d}\mathbb{E}\left[\frac{(e_{j}^{\top} X_{1})^2}{ a_{j}^2} \mathbf{1}\{(e_{j}^{\top} X_{1})^{2} > a_{j}^{2} n\}\right] + n^2\|\mathbb{E}[Y_1]\|_\infty^2.
      \end{eqnarray*} Then, there exists a universal constant $C>0$ such that
      \begin{eqnarray*}
          \sup_{t \in \mathbb{R}}\, |\mathbb{P}(\|Z\|_\infty \le t) - \mathbb{P}(\|Z^X\|_\infty \le t)| \le C\log(ed)R_n^{1/2}.
      \end{eqnarray*} Furthermore, let $\lambda_{\min}(\Omega^X)$ be the smallest eigenvalue of $\Omega^X$. Then, we have an improved upper bound as
    \begin{align*}
        \sup_{t \in \mathbb{R}}\, |\mathbb{P}(\|Z\|_\infty \le t) - \mathbb{P}(\|Z^X\|_\infty \le t)| \le \frac{C}{\lambda_{\min}(\Omega^X)}\log(ed) R_n\Set{\log(R_n^{-1})\vee1} ,
    \end{align*} where $C$ represents a possibly different universal constant.
\end{lemma}
\begin{proof}[\bfseries{Proof of \Cref{lemma:remainder3}}]
The results are consequences of two sharp Gaussian comparison inequalities. We recall that the diagonal entries of $\Omega^X$ and $\Omega:={\rm Var}(Z)$ are 1's. Define
\begin{eqnarray*}
    \varpi_n := \max_{1\leq j,k\leq d}\Abs{e_j^\top\left(\Omega-\Omega^X\right)e_k}.
\end{eqnarray*} Proposition~2.1 of \cite{chernozhuokov2022improved} shows that
\begin{eqnarray*}
    \sup_{t \in \mathbb{R}}\, |\mathbb{P}(\|Z\|_\infty \le t) - \mathbb{P}(\|Z^X\|_\infty \le t)| \le C\log(ed)\varpi_n^{1/2},
\end{eqnarray*} for some universal constant $C>0$. Meanwhile, Theorem 2.3 of \citet{lopes2022central} states that
\begin{eqnarray*}
    \sup_{t \in \mathbb{R}}\, |\mathbb{P}(\|Z\|_\infty \le t) - \mathbb{P}(\|Z^X\|_\infty \le t)| \le C\lambda^{-1}_{\min}(\Omega^X)\log(ed)\varpi_n\Set{\log(1/\varpi_n)\vee1},
\end{eqnarray*} where $C$ is a universal constant. Hence, in order to employ two bounds, it suffices to control $\varpi_n$. To this end, we first note that
\begin{align}\label{eq:pf.lem5:1}
    \varpi_n &= n\max_{1\leq j, k \leq p}\Abs{e_j^\top\left({\rm Var}(Z_1)-{\rm Var}(Z_1^X)\right)e_k}\nonumber\\
    &=n\max_{1\leq j,k\leq d}\Abs{\mathbb{E}[e_j^\top Z^X_1 e_k^\top Z^X_1] - \mathbb{E}[e_j^\top Y_1e_k^\top Y_1] + \mathbb{E}[e_j^\top Y_1] \E[e_k^\top Y_1]} \nonumber\\
    & \le n\max_{1\leq j,k\leq d}\left|\mathbb{E}[e_j^\top Z^X_1 e_k^\top Z^X_1] - \mathbb{E}[e_j^\top Y_1e_k^\top Y_1]\right| + n\max_{1\leq j,k\leq d}\Abs{\mathbb{E}[e_j^\top Y_1]\E[e_k^\top Y_1]}\nonumber\\
    &=n\max_{1\leq j,k\leq d}\left|\mathbb{E}[e_j^\top Z^X_1 e_k^\top Z^X_1] - \mathbb{E}[e_j^\top Y_1e_k^\top Y_1]\right| + n\|\E[Y_1]\|_\infty^2.
\end{align}
To analyze the leading term on the last display, we define 
\begin{eqnarray*}
    \mathcal{E}_{j,k}=\Set{(e_{j}^{\top} X_{1})^{2} \leq a_{j}^{2} n} \cap \Set{(e_{k}^{\top} X_{1})^{2} \leq a_{k}^{2} n}\quad\mbox{for}\quad 1\leq j,k\leq d.
\end{eqnarray*} Note that one has $e_j^\top Y_1=e_j^\top X_1/(n^{1/2}a_j)$ and $e_k^\top Y_1=e_k^\top X_1/(n^{1/2}a_k)$ on $\Ec_{jk}$. We observe
\begin{align}\label{eq:pf.lem5:2}
    &|\mathbb{E}[e_j^\top Z^X_1 e_k^\top Z^X_1] - \mathbb{E}[e_j^\top Y_1e_k^\top Y_1]| \nonumber\\
    &\qquad = |\mathbb{E}[(e_j^\top Z^X_1 e_k^\top Z^X_1-e_j^\top Y_1e_k^\top Y_1) \mathbf{1}(\mathcal{E}_{j,k})]| + |\mathbb{E}[(e_j^\top Z^X_1 e_k^\top Z^X_1-e_j^\top Y_1e_k^\top Y_1)\mathbf{1}(\mathcal{E}_{j,k}^\complement)]|.
\end{align} The first term is bounded by
\begin{align*}
    \left|\mathbb{E}\left[\frac{e_j^\top X_1 e_k^\top X_1}{n\sigma_j\sigma_k} -\frac{e_{j}^{\top} X_{1}e_{k}^{\top} X_{1}}{n a_{j} a_{k}}\right]\right| = \frac{|a_ja_k-\sigma_j\sigma_k|}{na_ja_k\sigma_j\sigma_k}|\mathbb{E}[e_j^\top X_1 e_k^\top X_1]| \le n^{-1}\left|1-\frac{\sigma_j\sigma_k}{a_ja_k}\right|.
\end{align*}
The last expression is maximized for some $j=k$ over $1\le j,k\le d$. Hence 
\begin{align}\label{eq:pf.lem5:3}
    \max_{1\le j,k\le d}|\mathbb{E}[(e_j^\top Z^X_1 e_k^\top Z^X_1-e_j^\top Y_1e_k^\top Y_1)\,  \mathbf{1}(\mathcal{E}_{j,k})]| \le \max_{1\le j \le d}n^{-1}\left|1-\frac{\sigma_j^2}{a_j^2}\right|.
\end{align}
On the other hand, the second term is bounded by
\begin{align*}
    &\left|\mathbb{E}\left[\left(\frac{e_j^\top X_1 e_k^\top X_1}{n\sigma_j\sigma_k} -\frac{e_{j}^{\top} X_{1}e_{k}^{\top} X_{1}}{n a_{j} a_{k}}\right)(\mathbf{1}\{(e_{k}^{\top} X_{1})^{2} > a_{k}^{2} n\} + \mathbf{1}\{(e_{j}^{\top} X_{1})^{2} > a_{j}^{2} n\})\right]\right| \\
    &\qquad \le \mathbb{E}\left[\left|\frac{e_j^\top X_1 e_k^\top X_1}{n\sigma_j\sigma_k} \right|(\mathbf{1}\{(e_{k}^{\top} X_{1})^{2} > a_{k}^{2} n\} + \mathbf{1}\{(e_{j}^{\top} X_{1})^{2} > a_{j}^{2} n\})\right] \\
    &\qquad\qquad + \mathbb{E}\left[\left|\frac{e_{j}^{\top} X_{1}e_{k}^{\top} X_{1}}{n a_{j} a_{k}} \right|(\mathbf{1}\{(e_{k}^{\top} X_{1})^{2} > a_{k}^{2} n\} + \mathbf{1}\{(e_{j}^{\top} X_{1})^{2} > a_{j}^{2} n\})\right]\\
    &\qquad 
    \le 2\left(1 \vee\frac{\sigma_j\sigma_k}{a_{j} a_{k}}\right)\mathbb{E}\left[\left|\frac{e_{j}^{\top} X_{1}e_{k}^{\top} X_{1}}{n a_{j} a_{k}} \right|(\mathbf{1}\{(e_{k}^{\top} X_{1})^{2} > a_{k}^{2} n\} + \mathbf{1}\{(e_{j}^{\top} X_{1})^{2} > a_{j}^{2} n\})\right].
\end{align*}
Similarly, the last expression is maximized for some $j=k$ over $1\le j,k\le d$. Hence 
\begin{align}\label{eq:pf.lem5:4}
    &\max_{1\le j, k \le d}|\mathbb{E}[(e_j^\top Z^X_1 e_k^\top Z^X_1-e_j^\top Y_1e_k^\top Y_1)\, \mathbf{1}\{ \mathcal{E}^\complement_{j,k}\}]| \nonumber\\
    &\qquad \le  \max_{1\le j \le d}4\left(1 \vee\frac{\sigma_j^2}{a_{j}^2}\right)\mathbb{E}\left[\left|\frac{(e_{j}^{\top} X_{1})^2}{n a_{j}^2} \right|\mathbf{1}\{(e_{j}^{\top} X_{1})^{2} > a_{j}^{2} n\}\right].
\end{align} Combining \eqref{eq:pf.lem5:1}, \eqref{eq:pf.lem5:2}, \eqref{eq:pf.lem5:3}, and \eqref{eq:pf.lem5:4}, we get
\begin{eqnarray*}
    \varpi_n\leq \max_{1\leq j\leq d}\Abs{1-\frac{\sigma_j^2}{a_j^2}}+\max_{1\le j \le d}4\left(1 \vee\frac{\sigma_j^2}{a_{j}^2}\right)\mathbb{E}\left[\left|\frac{(e_{j}^{\top} X_{1})^2}{a_{j}^2} \right|\mathbf{1}\{(e_{j}^{\top} X_{1})^{2} > a_{j}^{2} n\}\right] +n^2\Norm{\Eb Y_1}_\infty^2.
\end{eqnarray*}
To get the desired bound in Lemma~\ref{lemma:remainder3}, consider the case where $\sigma_j^2/2 > a_j^2$. It follows from the definition of $a_j$ that
\begin{align*}
   \sigma^2_j/2 > \E[(e_{j}^{\top} X_{1})^{2} \mathbf{1}\{(e_{j}^{\top} X_{1})^{2} \leq a_{j}^{2} n\}] &= \sigma^2_j\left(1-\frac{1}{2}\E\left[\frac{(e_{j}^{\top} X_{1})^{2}}{\sigma^2_j/2} \mathbf{1}\{(e_{j}^{\top} X_{1})^{2} > a_{j}^{2} n\}\right]\right)\\
   &\ge \sigma^2_j\left(1-\frac{1}{2}\E\left[\frac{(e_{j}^{\top} X_{1})^{2}}{a_j} \mathbf{1}\{(e_{j}^{\top} X_{1})^{2} > a_{j}^{2} n\}\right]\right).
\end{align*} This implies that
\begin{align*}
    R_n \ge \E\left[\frac{(e_{j}^{\top} X_{1})^{2}}{a_j} \mathbf{1}\{(e_{j}^{\top} X_{1})^{2} > a_{j}^{2} n\}\right] > 1.
\end{align*} Hence, Lemma~\ref{lemma:remainder3} follows by taking sufficiently large absolute constant $C>0$. Now, we consider the case $\sigma_j^2/2 \le a_j^2$, which follows as
\begin{align*}
    |1-\sigma_j^2/a_j^2|\le \frac{2|a_j^2 - \sigma_j^2|}{\sigma_j^2} &= \frac{2}{\sigma^2_j}\left|\E[(e_{j}^{\top} X_{1})^{2} \mathbf{1}\{(e_{j}^{\top} X_{1})^{2} \leq a_{j}^{2} n\}]-\E[(e_{j}^{\top} X_{1})^{2}]\right|\\
    &= \frac{2}{\sigma^2_j}\left|\E[(e_{j}^{\top} X_{1})^{2} \mathbf{1}\{(e_{j}^{\top} X_{1})^{2} > a_{j}^{2} n\}]\right|\\
    &\le 2\E\left[\frac{(e_{j}^{\top} X_{1})^{2}}{a_j^2} \mathbf{1}\{(e_{j}^{\top} X_{1})^{2} > a_{j}^{2} n\}\right]
\end{align*}
where the last step follows from the fact that $a_j^2 \le \sigma_j^2$ by \Cref{lemma:truncation}. Hence, the quantity $\varpi_n$ can be further bounded as
\begin{equation*}
    \varpi_n\leq 6\max_{1\le j \le d}\mathbb{E}\left[\left|\frac{(e_{j}^{\top} X_{1})^2}{a_{j}^2} \right|\mathbf{1}\{(e_{j}^{\top} X_{1})^{2} > a_{j}^{2} n\}\right] +n^2\Norm{\Eb Y_1}_\infty^2\leq 6R_n.
\end{equation*} This completes the proof.
\end{proof}
\begin{lemma}\label{lemma:Y_Ytilde}
    Recall $\Yv = \sum_{i=1}^n Y_i$ and $\widetilde{\Yv} = \sum_{i=1}^n \widetilde{Y}_i$ where $Y_i$ and $\widetilde{Y}_i$ are defined in \eqref{def:truncted_rv} and \eqref{def:Ytilde}. For $H = H_{\varepsilon, r} : \mathbb{R}^d \mapsto \mathbb{R}$, defined as \eqref{def:H-smooth}, we let
    \begin{eqnarray}\label{def:hj}
    h_j = \sup_{x}\sup_{\norm{u_1\otimes\cdots\otimes u_j}_\infty\leq1}\Abs{\Ip{\nabla^jH(x),u_1\otimes\cdots\otimes u_j}},
\end{eqnarray} for $j\in\mathbb{N}$. Suppose that $n\norm{\Eb Y_1}_\infty\leq1$ and $n\log^2(ed)\Eb[\norm{Y_1}_\infty^3]\leq1$, then there exists a universal constant $C>0$ such that
    \begin{align*}
        &|\Eb[H(\widetilde\Yv)]-\Eb\left[H(\Yv)\right]|\\
        &\qquad \le C\Big[h_1\left(n\log(ed)\Eb[\norm{Y_1}_\infty^3+n\Norm{\Eb[Y_1]}_\infty\right)+h_2\left(n\log^2(ed)\Eb[\norm{Y_1}_\infty^3] + n^2\Norm{\Eb[Y_1]}_\infty^2\right)\\
        &\qquad\qquad+h_3n\Eb[\Norm{Y_1}_\infty^3]+h_4\left(n\Eb[\Norm{Y_1}_\infty^3]\right)^2\Big].
    \end{align*}
\end{lemma}
\begin{proof}[\bfseries{Proof of \Cref{lemma:Y_Ytilde}}]

It follows from the definition of $\widetilde{Y}_i$ and the first-order Taylor series expansion of $g(x)$ at $x=1$ (recalling $g(1)=1$),
\begin{align}
    \widetilde\Yv&=\Yv + \sum_{j=1}^d e_j^\top \Yv\set{g(1+\eta_{j,n})-1}e_j=\Yv + \sum_{j=1}^d\int_0^1 e_j^\top \Yv\eta_{j,n}g'(1+\theta_1\eta_{j,n})e_j\,d\theta_1.\label{eq:step4:1}
\end{align}
Recall that $\eta_{j, n}=\sum_{i=1}^{n}\{(e_{j}^{\top} Y_{i})^{2}-1/n\}$. Define mappings 
\[L_j(y) = (e_j^\top y)^3-n^{-1}(e_j^\top y) \quad\mbox{and}\quad U_j(y_1,y_2)=(e_j^\top y_1)\{(e_j^\top y_2)^2-n^{-1}\}\]
for all $j=1,\ldots,d$. Then, one can write
\begin{eqnarray*}
    \sum_{i=1}^n(e_j^\top Y_i)\eta_{j,n} = \sum_{i=1}^nL_j(Y_i) + \sum_{i_1\neq i_2}U_j(Y_{i_1},Y_{i_2}).
\end{eqnarray*} Combining this with \eqref{eq:step4:1} gives
\begin{align*}
    \widetilde\Yv-\Yv= I_1 + I_2\quad\mbox{where}\quad I_1&:=\sum_{i=1}^n\sum_{j=1}^d\int_0^1L_j(Y_i)g'(1+\theta_1\eta_{j,n})e_j\,d\theta_1,\\
     \mbox{and}\quad I_2&:= \sum_{i_1\neq i_2}\sum_{j=1}^d\int_0^1U_j(Y_{i_1},Y_{i_2})g'(1+\theta_1\eta_{j,n})e_j\,d\theta_1.
\end{align*}
Similarly, the first-order Taylor series expansion of $g'(x)$ at $x=1$ implies
\begin{align*}
    I_2= I_3+I_4\quad\mbox{where}\quad I_3&:=-\frac{1}{2}\sum_{i_1\neq i_2}\sum_{j=1}^dU_j(Y_{i_1},Y_{i_2})e_j\\
    \mbox{and}\quad I_4&:=\sum_{i_1\neq i_2}\sum_{j=1}^d\int_0^1\int_0^1U_j(Y_{i_1},Y_{i_2})\theta_1\eta_{j,n}g''(1+\theta_1\theta_2\eta_{j,n})e_j\,d\theta_1d\theta_2.
\end{align*}
The term of interest can be bounded as
\begin{align*}
     &\Abs{\Eb[H(\widetilde\Yv)-H\left(\Yv\right)]} \le  \Abs{\Eb[H(\widetilde\Yv)-H\left(\Yv+I_2\right)]} + \Abs{\Eb[H(\Yv + I_2)-H\left(\Yv\right)]}\\
     &\qquad \le \Abs{\Eb[H(\widetilde\Yv)-H\left(\Yv+I_2\right)]} + \Abs{\Eb[H(\Yv + I_2)-H\left(\Yv + I_3\right)]} + \Abs{\Eb[H(\Yv + I_3)-H\left(\Yv\right)]}\\
     &\qquad := \mathrm{R}_1 + \mathrm{R}_2 +\mathrm{R}_3.
\end{align*}
We analyze three terms separately. 
First, following the similar approach as the proof of Lemma~2.3 of  \cite{bentkus1996berry},
\begin{align*}
    \mathrm{R}_1 &= \Abs{\Eb[H(\widetilde\Yv)-H\left(\Yv+I_2\right)]}\leq\Eb\left[h_1\Norm{I_1}_\infty\right]\\
&=h_1\Eb\left[\max_{j=1,\ldots,d}\Abs{\sum_{i=1}^nL_j(Y_i)\int_0^1g'(1+\theta_1\eta_{j,n})\,d\theta_1}\right]\\
    &\leq4h_1\Eb\left[\max_{j=1,\ldots,d}\Abs{\sum_{i=1}^nL_j(Y_i)}\right]\leq 4h_1\left(\Eb\left[\Norm{Y_1}_\infty\right]+n\Eb[\Norm{Y_1}_\infty^3]\right)\\
    &\leq8h_1n\Eb[\Norm{Y_1}_\infty^3],
\end{align*}
where we have used $\norm{g'}_\infty\leq4$. The last step follows as 
\begin{align*}
    \Eb\left[\Norm{Y_1}_\infty\right]  =   (\Eb[(e_j^\top Y_1)^2])^{-1} \Eb[(e_j^\top Y_1)^2] \Eb\left[\Norm{Y_1}_\infty\right] \le n\Eb[\Norm{Y_1}_\infty^3] 
\end{align*}
since $\Eb[(e_j^\top Y_1)^2] = 1/n$ from the definition of $Y_1$. To control $\mathrm{R}_2$, we note that
\begin{align*}
    \mathrm{R}_2 &= \Abs{\Eb\left[H\left(\Yv+I_2\right)-H\left(\Yv+I_3\right)\right]} = \Abs{\Eb\left[\Ip{H\left(\Yv+I_3\right),I_4}\right]}\leq\Eb\left[h_1\Norm{I_4}_\infty\right]\\
&=h_1\Eb\left[\max_{j=1,\ldots,d}\Abs{\sum_{i_1\neq i_2}U_j(Y_{i_1},Y_{i_2})\eta_{j,n}\int_0^1\int_0^1\theta_1g''(1+\theta_1\theta_2\eta_{j,n})\,d\theta_1d\theta_2}\right]\\
&\leq12h_1\Eb\left[\max_{j=1,\ldots,d}\Abs{\sum_{i_1\neq i_2}U_j(Y_{i_1},Y_{i_2})\eta_{j,n}}\right]\\
    &\leq12h_1\left\{\Eb\left[\max_{j=1,\ldots,d}\Abs{\sum_{i_1\neq i_2}U_j(Y_{i_1},Y_{i_2})}^2\right]\right\}^{1/2}\left\{\Eb\left[\max_{j=1,\ldots,d}\Abs{\sum_{i=1}^n(e_j^\top Y_i)^2-1}^2\right]\right\}^{1/2},
\end{align*}
where we used $\norm{g''}_\infty\leq 24$ for the penultimate inequality and the last inequality is Cauchy Schwarz inequality. \Cref{lemma:remainder2} proves that
\begin{align*}
    \left\{\Eb\left[\max_{j=1,\ldots,d}\Abs{\sum_{i=1}^n(e_j^\top Y_i)^2-1}^2\right]\right\}^{1/2} \le Cn^{1/2}\log^{1/2}(ed)\left(\Eb[\norm{Y_1}_\infty^3]\right)^{1/2},
\end{align*}
with some universal constant $C > 0$. Moreover, \Cref{lemma:u-statistics} implies that 
\begin{align*}
    \left\{\Eb\left[\max_{j=1,\ldots,d}\Abs{\sum_{i_1\neq i_2}U_j(Y_{i_1},Y_{i_2})}^2\right]\right\}^{1/2} \le C\left(n^{1/2}\log(ed)\left(\Eb[\norm{Y_1}_\infty^3]\right)^{1/2} + n\Norm{\Eb[Y_1]}_\infty\right),
\end{align*} for a possibly different universal constant $C>0$. Hence, as long as $n\log^2(d)\Eb[\norm{Y_1}_\infty^3]\leq1$, we get
\begin{eqnarray*}
    \mathrm{R}_2\leq C h_1\left(n\log(ed)\Eb[\norm{Y_1}_\infty^3+n\Norm{\Eb[Y_1]}_\infty\right).
\end{eqnarray*} The quantity $\mathrm{R}_3$ can be bounded as:
\begin{align*}
    \mathrm{R}_3 &=\Abs{\Eb[H(\Yv + I_3)-H\left(\Yv\right)]} \\&=\Abs{\Eb\left[H\left(\Yv+I_3\right)-H\left(\Yv\right)-\left\langle\nabla H\left(\Yv\right),I_3\right\rangle + \left\langle\nabla H\left(\Yv\right),I_3\right\rangle\right]} \\
    &\le\Abs{\Eb[H(\Yv+I_3)-H(\Yv)-\langle\nabla H(\Yv),I_3\rangle]} + \Abs{\Eb[\langle\nabla H(\Yv),I_3\rangle]}\\
    &=:\mathrm{R}_4+\mathrm{R}_5.
\end{align*}
To bound $\mathrm{R}_4$, it follows by the second order Taylor expansion of $H$ at $\Yv$ that
\begin{align*}
    \mathrm{R}_4&=\Abs{\Eb\left[(1-\tau)\Ip{\nabla^2H\left(\Yv+\tau I_3\right),I_3^{\otimes2}}\right]}\\
    &\leq\frac{1}{2}\Eb\left[h_2\norm{I_3}_\infty^2\right] = \frac{h_2}{8}\Eb\left[\max_{j=1,\ldots,d}\Abs{\sum_{i_1\neq i_2}U_j(Y_{i_1},Y_{i_2})}^2\right],
\end{align*}
where $\tau\sim{\rm unif}(0,1)$, independent of everything. Hence, it can be further bounded with the aid of \Cref{lemma:u-statistics} as
\begin{eqnarray*}
    \mathrm{R}_4\leq Ch_2\left(n\log^2(ed)\Eb[\norm{Y_1}_\infty^3] + n^2\Norm{\Eb[Y_1]}_\infty^2\right).
\end{eqnarray*}
Finally, to control $\mathrm{R}_5$, we begin by noting that
\begin{align}\label{eq:R_5}
    \mathrm{R}_5
    & = \frac{1}{2}\Abs{\Eb\left[\Ip{\nabla H\left(\Yv\right),\sum_{i_1\neq i_2}\sum_{j=1}^dU_j(Y_{i_1},Y_{i_2})e_j}\right]}\nonumber\\
    &=\frac{n(n-1)}{2}\Abs{\Eb\left[\Ip{\nabla H\left(\Yv\right),\sum_{j=1}^dU_j(Y_1,Y_2)e_j}\right]}.
\end{align} To control the last expression, we write $\Uv(Y_1,Y_2) = \sum_{j=1}^dU_j(Y_1,Y_2)e_j$ and note from the first order Taylor expansion of $\nabla H$ at $\Yv_{-1}:=\Yv-Y_1$ that
\begin{align}
    \Ip{\nabla H\left(\Yv\right),\Uv(Y_1,Y_2)}
    &= \Ip{\nabla H\left(\Yv_{-1}\right),\Uv(Y_1,Y_2)}\label{eq:pf.lem5.1}\\
    &\qquad+ \,\Eb_{\tau_1}\left[\Ip{\nabla^2 H\left(\Yv_{-1}+\tau_1 Y_1\right),\Uv(Y_1,Y_2)\otimes Y_1}\right],\label{eq:pf.lem5.2}
\end{align} where $\tau_1\sim{\rm Unif}(0,1)$ and $\Eb_{\tau_1}$ is the expectation taken over $\tau_1$. For \eqref{eq:pf.lem5.1}, we do Taylor expansion once more on $\nabla H$ at $\Yv_{-\set{1,2}}:=\Yv-Y_1-Y_2$, which leads to
\begin{align*}
    \Ip{\nabla H\left(\Yv_{-1}\right),\Uv(Y_1,Y_2)} &= \Ip{\nabla H\left(\Yv_{-\set{1,2}}\right),\Uv(Y_1,Y_2)} \\
    &\qquad +\Eb_{\tau_2}\left[\Ip{\nabla^2 H\left(\Yv_{-\set{1,2}}+\tau_2 Y_2\right),\Uv(Y_1,Y_2)\otimes Y_2}\right],
\end{align*} for an independent $\tau_2\sim{\rm Unif}(0,1)$. Since $\Yv_{-\set{1,2}}$ and $(Y_1,Y_2)$ are independent and $\Eb[\Uv(Y_1,Y_2)]=0$, we have
\begin{eqnarray*}
    \Eb\left[\Ip{\nabla H\left(\Yv_{-\set{1,2}}\right),\Uv(Y_1,Y_2)}\right]=\Ip{\Eb\left[\nabla H\left(\Yv_{-\set{1,2}}\right)\right],\Eb\left[\Uv(Y_1,Y_2)\right]}=0.
\end{eqnarray*} Moreover, since $Y_1$ is independent from $\Yv_{-\set{1,2}}$ and $Y_2$,
\begin{align*}
    &\Abs{\Eb\left[\Ip{\nabla^2 H\left(\Yv_{-\set{1,2}}+\tau_2 Y_2\right),\Uv(Y_1,Y_2)\otimes Y_2}\right]}\\
    &\qquad=\Abs{\Eb\left[\Ip{\nabla^2 H\left(\Yv_{-\set{1,2}}+\tau_2 Y_2\right),\Eb_{Y_1}\left[\Uv(Y_1,Y_2)\right]\otimes Y_2}\right]}
    \\
    &\qquad\leq h_2\Eb\left[\Norm{\Eb_{Y_1}[\Uv(Y_1,Y_2)\otimes Y_2]}_\infty\right]\\
    &\qquad\leq h_2\Norm{\Eb[Y_1]}_\infty\Eb\left[\max_{j=1,\ldots,d}\Abs{(e_j^\top Y_2)^2-\frac{1}{n}} \times\max_{j=1,\ldots,d}\Abs{e_j^\top Y_2}\right]\\
    &\qquad\leq h_2\Norm{\Eb[Y_1]}_\infty\Set{\Eb\left[\Norm{Y_1}_\infty^3 + n^{-1}\Eb\left[\Norm{Y_1}_\infty\right]\right]}\leq 2h_2\Norm{\Eb[Y_1]}_\infty\Eb[\Norm{Y_1}_\infty^3],
\end{align*} where the last inequality follows from $\Eb[\norm{Y_1}_\infty]\leq n\Eb[\norm{Y_1}_\infty^2]$. Combining these gives the bound for the expected value of \eqref{eq:pf.lem5.1} as
\begin{eqnarray*}
    \Abs{\Ip{\nabla H\left(\Yv_{-1}\right),\Uv(Y_1,Y_2)}}\leq 2h_2\Norm{\Eb[Y_1]}_\infty\Eb[\Norm{Y_1}_\infty^3].
\end{eqnarray*} To control \eqref{eq:pf.lem5.2}, we inspect the quantity inside the expectation and apply Taylor expansion on $\nabla H$ at $\Yv_{-\set{1,2}}+\tau Y_1$ for given $\tau_1\in(0,1)$ as
\begin{align}\label{eq:pf.lem5.3}
    &\Ip{\nabla^2 H\left(\Yv_{-1}+\tau_1 Y_1\right),\Uv(Y_1,Y_2)\otimes Y_1}\nonumber\\
    &\qquad =\Ip{\nabla^2 H\left(\Yv_{-\set{1,2}}+\tau_1 Y_1\right),\Uv(Y_1,Y_2)\otimes Y_1}\nonumber\\
    &\qquad\qquad+\Eb_{\tau_3}\left[\Ip{\nabla^3 H\left(\Yv_{-\set{1,2}}+\tau_1 Y_1+\tau_3 Y_2\right),\Uv(Y_1,Y_2)\otimes Y_1\otimes Y_2}\right],
\end{align} for $\tau_3\sim{\rm Unif}(0,1)$. The independence of $Y_2$ and $(\Yv_{-\set{1,2}}, Y_1)$ helps control the leading term on the right-hand side of \eqref{eq:pf.lem5.3} as
\begin{align*}
    &\Eb\left[\Ip{\nabla^2 H\left(\Yv_{-\set{1,2}}+\tau_1 Y_1\right),\Uv(Y_1,Y_2)\otimes Y_1}\right]\\
    &\qquad=\Eb\left[\Ip{\nabla^2 H\left(\Yv_{-\set{1,2}}+\tau_1 Y_1\right),\Eb_{Y_2}[\Uv(Y_1,Y_2)]\otimes Y_1}\right]=0.
\end{align*} For the rightmost term in \eqref{eq:pf.lem5.3}, we apply Taylor expansion once more on $\nabla^3 H$ at $\Yv_{-\set{1,2}}+\tau_3Y_2$ for a fixed $\tau_2\in(0,1)$ to get
\begin{align*}
     &\Eb\left[\Ip{\nabla^3 H\left(\Yv_{-\set{1,2}}+\tau_1 Y_1+\tau_3 Y_2\right),\Uv(Y_1,Y_2)\otimes Y_1\otimes Y_2}\right]\\
     &\qquad=\Eb\left[\Ip{\nabla^3 H\left(\Yv_{-\set{1,2}}+\tau_3 Y_2\right),\Uv(Y_1,Y_2)\otimes Y_1\otimes Y_2}\right]\\
     &\qquad\qquad+\Eb\left[\Ip{\nabla^4 H\left(\Yv_{-\set{1,2}}+\tau_1\tau_4 Y_1+\tau_3 Y_2\right),\Uv(Y_1,Y_2)\otimes Y_1\otimes Y_2\otimes(\tau_1Y_1)}\right].
\end{align*} Note that
\begin{align*}
    &\Abs{\Eb\left[\Ip{\nabla^3 H\left(\Yv_{-\set{1,2}}+\tau_3 Y_2\right),\Uv(Y_1,Y_2)\otimes Y_1\otimes Y_2}\right]} \\
    &\qquad= \Abs{\Eb\left[\Ip{\nabla^3 H\left(\Yv_{-\set{1,2}}+\tau_3 Y_2\right),\Eb_{Y_1}[\Uv(Y_1,Y_2)\otimes Y_1]\otimes Y_2}\right]}\\
    &\qquad\leq h_3\Eb\left[\Norm{\Eb_{Y_1}[\Uv(Y_1,Y_2)\otimes Y_1]\otimes Y_2}_\infty\right]\\
    &\qquad\leq h_3\left(\Norm{\Eb Y_1}_\infty^2\vee\frac{1}{n}\right)\Eb\left[\max_{j=1,\ldots,d}\Abs{(e_j^\top Y_2)^2-\frac{1}{n}} \max_{j=1,\ldots,d}\Abs{e_j^\top Y_2}\right]\\
    &\qquad\leq2h_3n^{-1}\Eb[\Norm{Y_1}_\infty^3].
\end{align*} Furthermore, we have
\begin{align*}
    &\Eb\left[\Ip{\nabla^4 H\left(\Yv_{-\set{1,2}}+\tau_1\tau_4 Y_1+\tau_3 Y_2\right),\Uv(Y_1,Y_2)\otimes Y_1\otimes Y_2\otimes(\tau_1Y_1)}\right]\\
    &\qquad \leq h_4\Eb\left[\norm{\Uv(Y_1,Y_2)\otimes Y_1\otimes Y_2\otimes(\tau_1Y_1)}_\infty\right]\\
    &\qquad\leq\frac{h_4}{2}\Eb\left[\norm{Y_1}_\infty^3\right]\Eb\left[\max_{j=1,\ldots,d}\Abs{(e_j^\top Y_2)^2-\frac{1}{n}} \max_{j=1,\ldots,d}\Abs{e_j^\top Y_2}\right]\\
    &\qquad\leq h_4\left(\Eb[\Norm{Y_1}_\infty^3]\right)^2.
\end{align*} Combining these controls the last term in \eqref{eq:pf.lem5.3} as
\begin{align*}
    \Eb\left[\Ip{\nabla^3 H\left(\Yv_{-\set{1,2}}+\tau_1 Y_1+\tau_3 Y_2\right),\Uv(Y_1,Y_2)\otimes Y_1\otimes Y_2}\right]
    \leq 2h_3n^{-1}\Eb[\Norm{Y_1}_\infty^3]+h_4\left(\Eb[\Norm{Y_1}_\infty^3]\right)^2.
\end{align*} Putting all together leads to
\begin{align*}
    &\Abs{\Eb\Ip{\nabla H\left(\Yv\right),\Uv(Y_1,Y_2)}}\leq 2h_2\Norm{\Eb[Y_1]}_\infty\Eb[\Norm{Y_1}_\infty^3] +2h_3n^{-1}\Eb[\Norm{Y_1}_\infty^3]+h_4\left(\Eb[\Norm{Y_1}_\infty^3]\right)^2.
\end{align*} This implies that
\begin{align*}
    \mathrm{R}_5 \leq h_2n\Norm{\Eb[Y_1]}_\infty n\Eb[\Norm{Y_1}_\infty^3]+h_3n\Eb[\Norm{Y_1}_\infty^3]+h_4\left(n\Eb[\Norm{Y_1}_\infty^3]\right)^2.
\end{align*} Hence, as long as $n\Eb[\norm{Y_1}_\infty]^3\leq1$, one has
\begin{eqnarray*}
    \mathrm{R}_5\leq h_2n\Norm{\Eb[Y_1]}_\infty+h_3n\Eb[\Norm{Y_1}_\infty^3]+h_4\left(n\Eb[\Norm{Y_1}_\infty^3]\right)^2.
\end{eqnarray*} Putting all together concludes the proof.

    
\end{proof}

\begin{lemma}\label{lemma:u-statistics}
Recall $U_j (y_1,y_2) \mapsto (e_j^\top y_1)\{(e_j^\top y_2)^2-n^{-1}\}$ and $Y_i$ is defined as \eqref{def:truncted_rv}. Suppose that $n\log^2(ed)\Eb[\norm{Y_1}_\infty^3]\leq1$ holds, then there exists a universal constant $C>0$ such that 
    \begin{align}
        \Eb\left[\max_{j=1,\ldots,d}\Abs{\sum_{i_1\neq i_2}U_j(Y_{i_1},Y_{i_2})}^2\right] \le C\left(n\log^2(ed)\Eb[\norm{Y_1}_\infty^3] + n^2\Norm{\Eb[Y_1]}_\infty^2\right).
    \end{align}
\end{lemma}
\begin{proof}[\bfseries{Proof of \Cref{lemma:u-statistics}}]
    Let $(Y'_1,\ldots,Y_n')$ be an independent copy of $(Y_1,\ldots,Y_n)$. Equation (3.1.8) of \cite{de2012decoupling} on decoupling of $U$-statistics applies and yields
    \begin{align*}
&\Eb\left[\max_{j=1,\ldots,d}\Abs{\sum_{i_1\neq i_2}U_j(Y_{i_1},Y_{i_2})}^2\right]\leq 64\Eb\left[\max_{j=1,\ldots,d}\Abs{\sum_{i_1\neq i_2}U_j(Y_{i_1},Y'_{i_2})}^2\right]\\
    &\qquad =64\Eb\left[\max_{j=1,\ldots,d}\Abs{\sum_{i_1=1}^n\sum_{i_2=1}^n (e_j^\top Y_{i_1}) \left((e_j^\top Y_{i_2}')^2-\frac{1}{n}\right)-\sum_{i=1}^n\Set{(e_j^\top Y_i)^3-\frac{e_j^\top Y_i}{n}}}^2\right]\\
    &\qquad \leq128\Bigg[\Eb\left[\max_{j=1,\ldots,d}\Abs{\sum_{i=1}^n(e_j^\top Y_i)}^2\right]\Eb\left[\max_{j=1,\ldots,d}\Abs{\sum_{i=1}^n\Bigg\{(e_j^\top Y_i)^2-1/n\bigg\}}^2\right]\\
&\qquad \qquad +\Eb\left[\max_{j=1,\ldots,d}\Abs{\sum_{i=1}^n\Set{(e_j^\top Y_i)^3-\frac{e_j^\top Y_i}{n}}}^2\right]\Bigg].
    \end{align*}
Finally by \cref{lemma:moment-bounds-for-Ustats}, we conclude with some universal constant $C$, 
\begin{align*}
    &\Eb\left[\max_{j=1,\ldots,d}\Abs{\sum_{i_1\neq i_2}U_j(Y_{i_1},Y_{i_2})}^2\right]\\
    &\qquad \leq C\left(\log(ed)+\left(n\log^2(ed)\Eb[\norm{Y_1}_\infty^3]\right)^{2/3}+\left(n\log(ed)\Eb[\norm{Y_1}_\infty^3]\right)^{1/2}+n^2\Norm{\Eb[Y_1]}_\infty^2\right)\\
    &\qquad\quad\times\left(n\log(ed)\Eb[\norm{Y_1}_\infty^3] + \left(n\log^{1/2}(ed)\Eb[\norm{Y_1}_\infty^3]\right)^{4/3}\right)+ C\left(n\log(ed)\Eb[\norm{Y_1}_\infty^3] + n^2\Norm{\Eb Y_1}_\infty^2\right)\\
    &\qquad\leq C\Big[n\log^2(ed)\Eb[\norm{Y_1}_\infty^3] + \left(n\log^{5/4}(ed)\Eb[\norm{Y_1}_\infty^3]\right)^{4/3}+ \left(n\log^{7/5}(ed)\Eb[\norm{Y_1}_\infty^3]\right)^{5/3}\\
    &\qquad\qquad+\left(n\log(ed)\Eb[\norm{Y_1}_\infty^3]\right)^{2}+ \left(n\log(ed)\Eb[\norm{Y_1}_\infty^3]\right)^{3/2}+ \left(n\log^{7/11}(ed)\Eb[\norm{Y_1}_\infty^3]\right)^{11/6}\Big]\\
&\qquad\qquad+Cn^2\Norm{\Eb[Y_1]}_\infty^2\left[1+n\log(ed)\Eb[\norm{Y_1}_\infty^3] + \left(n\log^{1/2}(ed)\Eb[\norm{Y_1}_\infty^3]\right)^{4/3}\right].
\end{align*} Hence, the lemma follows.
\end{proof}

\begin{lemma}\label{lemma:moment-bounds-for-Ustats}
Recall $Y_i$ is defined in \eqref{def:truncted_rv} and we denote by $\lesssim$ the boundedness up to a universal constant. Then the following statements hold:
\begin{equation}
    \Eb\left[\max_{j=1,\ldots,d}\Abs{\sum_{i=1}^ne_j^\top Y_i}^2\right]\lesssim\\
\log(ed)+\left(n\log^2(ed)\Eb[\norm{Y_1}_\infty^3]\right)^{\frac{2}{3}}+\left(n\log(ed)\Eb[\norm{Y_1}_\infty^3]\right)^{\frac{1}{2}}+n^2\Norm{\Eb[Y_1]}_\infty^2,
\end{equation}
\begin{equation}
    \Eb\left[\max_{j=1,\ldots,d}\Abs{\sum_{i=1}^n(e_j^\top Y_i)^2-1}^2\right]\lesssim n\log(ed)\Eb[\norm{Y_1}_\infty^3] + \left(n\log^{1/2}(ed)\Eb[\norm{Y_1}_\infty^3]\right)^{4/3},
\end{equation} and
\begin{equation}
    \Eb\left[\max_{j=1,\ldots,d}\Abs{\sum_{i=1}^n(e_j^\top Y_i)^3-\frac{e_j^\top Y_i}{n}}^2\right] \lesssim n\log(ed)\Eb[\norm{Y_1}_\infty^3] + n^2\Norm{\Eb Y_1}_\infty^2.
\end{equation}
\end{lemma}
\begin{proof}[\bfseries{Proof of \Cref{lemma:moment-bounds-for-Ustats}}]
    All three results can be unified such that $\Eb\left[\max_{j=1,\ldots,d}\Abs{\sum_{i=1}^n\xi_{ij}}^2\right]$ for $\xi_{ij}\in\set{e_j^\top Y_i, (e_j^\top Y_i)^2-1/n, (e_j^\top Y_i)^3-e_j^\top Y_i/n}$. For any choice of $\xi_{ij}$, it holds that 
\begin{align*}
    &\frac{1}{2}\Eb\left[\max_{j=1,\ldots,d}\Abs{\sum_{i=1}^n\xi_{ij}}^2\right]\leq \Eb\left[\max_{j=1,\ldots,d}\Abs{\sum_{i=1}^n\xi_{ij}-\Eb[\xi_{ij}]}^2\right]+\max_{j=1,\ldots,d}\Abs{\sum_{i=1}^n\Eb[\xi_{ij}]}^2
    \\&\qquad = \left(\Eb\left[\max_{j=1,\ldots,d}\sum_{i=1}^n\left(\xi_{ij}-\Eb[\xi_{ij}]\right)\right]\right)^2 + {\rm Var}\left[\max_{j=1,\ldots,d}\sum_{i=1}^n\left(\xi_{ij}-\Eb[\xi_{ij}]\right)\right]+\max_{j=1,\ldots,d}\Abs{\sum_{i=1}^n\Eb[\xi_{ij}]}^2\\
    &\qquad\leq\left(\Eb\left[\max_{j=1,\ldots,d}\sum_{i=1}^n\left(\xi_{ij}-\Eb[\xi_{ij}]\right)\right]\right)^2 + \Eb\left[\max_{j=1,\ldots,d}\sum_{i=1}^n\left(\xi_{ij}-\Eb[\xi_{ij}]\right)^2\right]\\
    &\qquad\qquad+ \max_{j=1,\ldots,d}\Eb\left[\sum_{i=1}^n\left(\xi_{ij}-\Eb[\xi_{ij}]\right)^2\right]+\max_{j=1,\ldots,d}\Abs{\sum_{i=1}^n\Eb[\xi_{ij}]}^2\\
    &\qquad\leq\left(\Eb\left[\max_{j=1,\ldots,d}\sum_{i=1}^n\left(\xi_{ij}-\Eb[\xi_{ij}]\right)\right] \right)^2+ \Eb\left[\max_{j=1,\ldots,d}\sum_{i=1}^n\Set{\left(\xi_{ij}-\Eb[\xi_{ij}]\right)^2-\Eb[\left(\xi_{ij}-\Eb[\xi_{ij}]\right)^2]}\right]\\
    &\qquad\qquad+ 2\max_{j=1,\ldots,d}\Eb\left[\sum_{i=1}^n\left(\xi_{ij}-\Eb[\xi_{ij}]\right)^2\right]+\max_{j=1,\ldots,d}\Abs{\sum_{i=1}^n\Eb[\xi_{ij}]}^2
\end{align*}
where the first inequality is the application of Theorem 11.1 of \citet{Boucheron2013concentration}.
Furthermore, by Proposition B.1 of \citet{kuchibhotla2022least}, it follows for any $q>1$,
\begin{align*}
\left(\Eb\left[\max_{j=1,\ldots,d}\Abs{\sum_{i=1}^n\left(\xi_{ij}-\Eb[\xi_{ij}]\right)}\right]\right)^2\lesssim n\log(1+d)\max_{j}{\rm Var}(\xi_{1j})+\left(n\log^{q-1}(1+d)\Eb[\norm{\xi_1}_\infty^q]\right)^{2/q}.   
\end{align*}
Similarly, for any $q>2$, Proposition B.1 of \citet{kuchibhotla2022least} implies
\begin{align*}
&\Eb\left[\max_{j=1,\ldots,d}\sum_{i=1}^n\Set{\left(\xi_{ij}-\Eb[\xi_{ij}]\right)^2-\Eb[\left(\xi_{ij}-\Eb[\xi_{ij}]\right)^2]}\right]\\
&\qquad\lesssim \left(n\log(1+d)\max_{j}\Eb(\xi_{1j}^4)\right)^{1/2}+\left(n\log^{q/2-1}(1+d)\Eb[\norm{\xi_1}_\infty^{q}]\right)^{2/q}.
\end{align*}
It remains to apply these bounds to each choice in  $\set{e_j^\top Y_i, (e_j^\top Y_i)^2-1/n, (e_j^\top Y_i)^3-e_j^\top Y_i/n}$. This yields the following results:
\begin{align*}
    &\Eb\left[\max_{j=1,\ldots,d}\Abs{\sum_{i=1}^ne_j^\top Y_i}^2\right]\\
    &\qquad \lesssim \log(1+d)+\left(n\log^2(1+d)\Eb[\norm{Y_1}_\infty^3]\right)^{2/3}+\left(n\log(1+d)\Eb[\norm{Y_1}_\infty^3]\right)^{1/2}+n^2\Norm{\Eb[Y_1]}_\infty^2.
\end{align*}
Similarly, 
\begin{align*}
    &\Eb\left[\max_{j=1,\ldots,d}\Abs{\sum_{i=1}^n(e_j^\top Y_i)^2-1}^2\right]\\
    &\qquad \lesssim
        \log(1+d)\max_{j}\sum_{i=1}^n\Eb[(e_j^\top Y_i)^4]+\left(n\log^{1/2}(1+d)\Eb[\norm{Y_1}_\infty^3]\right)^{4/3}+n\Eb[\norm{Y_1}_\infty^3]\\
        &\qquad \lesssim n\log(1+d)\Eb[\norm{Y_1}_\infty^3] + \left(n\log^{1/2}(1+d)\Eb[\norm{Y_1}_\infty^3]\right)^{4/3}.
\end{align*}
Finally, 
\begin{align*}
    &\Eb\left[\max_{j=1,\ldots,d}\Abs{\sum_{i=1}^n(e_j^\top Y_i)^3-\frac{e_j^\top Y_i}{n}}^2\right]\\
    &\qquad \lesssim n^2\Norm{\Eb Y_1}_\infty^2+n^2\max_{j}\Eb^2[\abs{e_j^\top Y_i}^3]+n\Eb[\norm{Y_1}_\infty^3]\\
        &\qquad\qquad +
        \log(1+d)\max_{j}\sum_{i=1}^n\Eb[(e_j^\top Y_i)^3]+\left(n\log^{q-1}(1+d)\Eb[\norm{Y_1}_\infty^{3q}]\right)^{2/q}\\
        &\qquad\lesssim n\log(1+d)\Eb[\norm{Y_1}_\infty^3] + n^2\Norm{\Eb Y_1}_\infty^2+n^2\max_{j}\Eb^2[\abs{e_j^\top Y_i}^3].
\end{align*} We further note that $\Eb[\abs{e_j^\top Y_i}^3]\leq \Eb[\abs{e_j^\top Y_i}^2]=1/n$, and thus,
\begin{eqnarray*}
    \Eb\left[\max_{j=1,\ldots,d}\Abs{\sum_{i=1}^n(e_j^\top Y_i)^3-\frac{e_j^\top Y_i}{n}}^2\right]\lesssim n\log(1+d)\Eb[\norm{Y_1}_\infty^3] + n^2\Norm{\Eb Y_1}_\infty^2.
\end{eqnarray*}
    This concludes the claim.
\end{proof}
\subsection{Gaussian approximation}
\begin{proposition}\label{prop:B.1}
    Let $X_1,\ldots,X_n$ be centered independent random vectors in $\Real^d$ where $X_i=(X_{i1},\ldots,X_{id})^\top$ for all $i=1,\ldots,n$. Suppose that there exists positive constants $b_1$ and $b_2$, and a sequence of positive reals $\set{B_n\geq1}$ such that
    \begin{equation}\label{eq:propB.1:1}
        b_1\leq \min_{1\leq j\leq d}\left(\frac{1}{n}\sum_{i=1}^n\Eb[X_{ij}^2]\right)^{1/2}\quad\mbox{and}\quad\max_{1\leq j\leq d}\left(\frac{1}{n}\sum_{i=1}^n\Eb[X_{ij}^4]\right)^{1/4}\leq b_2 B_n.
    \end{equation} Furthermore, suppose there exists a sequence of reals $\set{D_n\geq1}$ such that
    \begin{equation}\label{eq:propB.1:2}
        \max_{1\leq i\leq n}\left(\Eb[\norm{X_i}_\infty^q]\right)^{1/q}\leq D_n,
    \end{equation} for some $q>2$. Let $\Ac$ be a class of hyperrectangles in $\Real^d$ and let $G\sim\Nc(\mathbf{0},\Sigma_n)$ where $\Sigma_n={\rm Var}(n^{-1/2}\sum_{i=1}^nX_i)$. Then, there exists a constant $C=C(b_1,b_2,q)>0$ such that
    \begin{equation*}
        \sup_{A\in\Ac}\Abs{\Pb\left(\frac{1}{\sqrt{n}}\sum_{i=1}^nX_i\in A\right)-\Pb\left(G\in A\right)}\leq C\left[\left(\frac{B_n^4\log^5(d)}{n}\right)^{1/4}+\frac{D_n\log^{3/2}(d)}{n^{1/2-1/q}}\right].
    \end{equation*}
\end{proposition}

This proposition refines Theorem~2.5 of \cite{chernozhuokov2022improved} where the author set $D_n=B_n^2$. The current version is more informative since $B_n^2$ and $D_n$ scale differently in general.

\begin{proof}[\bfseries{Proof of Proposition~\ref{prop:B.1}}] We shall show that
\begin{equation*}
        \chi_n:=\sup_{y\in \Real^d}\Abs{\Pb\left(\frac{1}{\sqrt{n}}\sum_{i=1}^nX_i\preceq y\right)-\Pb\left(G\preceq y\right)}\leq C\left[\left(\frac{B_n^4\log^5(d)}{n}\right)^{1/4}+\frac{D_n\log^{3/2}(d)}{n^{1/2-1/q}}\right],
\end{equation*} which implies the desired result. Here, $x\preceq y$ denotes $e_j^\top x\leq e_j^\top y$ for all $j=1,\ldots,d$. We use Lemma~A.1 of \cite{chernozhuokov2022improved} which states that
\begin{align}\label{eq:pf_propA.1:1}
    \chi_n&\leq C\bigg[\phi\log^2(d)\sqrt{\Delta_1}+\phi^3\log^{7/2}(d)\Delta_1+\log(d)\sqrt{\Delta_2(\phi)}\nonumber\\
    &\qquad +\phi\log^{3/2}(d)\Delta_2(\phi)+\log^{3/2}(d)\sqrt{\Delta_3(\phi)}+\frac{\log^{1/2}(d)}{\phi}\bigg],
\end{align} for all $\phi>0$ where $C>0$ only depends on $b_1$ and
\begin{align*}
    \Delta_1 &= \frac{1}{n^2}\max_{1\leq j\leq d}\sum_{i=1}^n\Eb[X_{ij}^4],\quad \Delta_2(\phi)=\max_{1\leq j\leq d}\sum_{i=1}^n\Eb\left[Y_{ij}^2\mathbf{1}\set{\norm{Y_i}_\infty>(\phi\log(d))^{-1}}\right],\\
    \Delta_3(\phi) &= \Eb\left[\max_{1\leq i\leq n}\norm{Y_i}_\infty^2\mathbf{1}\set{\norm{Y_i}_\infty>(\phi\log(d))^{-1}}\right],
\end{align*} and $Y_i = (X_i-\tilde{X}_i)/\sqrt{n}$ for all $i=1,\ldots,n$ and $\set{\tilde X_1,\ldots,\tilde X_n}$ is an independent copy of $\set{X_1,\ldots,X_n}$. First, it is immediate that
\begin{equation}\label{eq:pf_propA.1:3}
    \Delta_1\leq n^{-1}B_n^2.
\end{equation} To analyze $\Delta_2(\phi)$, we first note that
\begin{equation}\label{eq:pf_propA.1:2}
    \max_{1\leq j\leq d}\frac{1}{n}\sum_{i=1}^n\Eb[Y_{ij}^4]\lesssim n^{-2}B_n^4,\quad\mbox{and}\quad \max_{1\leq i\leq n}\Eb\left[\norm{Y_i}_\infty^q\right]\lesssim n^{-q/2}D_n^q,
\end{equation} where $\lesssim$ denotes an inequality that holds up to constant depending only on $b_1,b_2$ and $q$. We observe that
\begin{align*}
    Y_{ij}^2\mathbf{1}\set{\norm{Y_i}_\infty>A}&=Y_{ij}^2\mathbf{1}\set{\abs{Y_{ij}}>A,\norm{Y_i}_\infty>A}+Y_{ij}^2\mathbf{1}\set{\abs{Y_{ij}}\leq A,\norm{Y_i}_\infty>A}\\
    &=Y_{ij}^2\mathbf{1}\set{\abs{Y_{ij}}>A}+Y_{ij}^2\mathbf{1}\set{\abs{Y_{ij}}\leq A,\norm{Y_i}_\infty>A}\\
    &\leq Y_{ij}^2\mathbf{1}\set{\abs{Y_{ij}}>A}+A^2\mathbf{1}\set{\norm{Y_i}_\infty>A},
\end{align*} for any $A>0$. Consequently, we have
\begin{align*}
    \Eb\left[Y_{ij}^2\mathbf{1}\set{\norm{Y_i}_\infty>A}\right]&\leq \Eb\left[Y_{ij}^2\mathbf{1}\set{\abs{Y_{ij}}>A}\right]+A^2\Pb\left(\norm{Y_i}_\infty>A\right)\\
    &\leq\Eb\left[\frac{Y_{ij}^4}{A^2}\mathbf{1}\set{\abs{Y_{ij}}>A}\right]+A^2\frac{\Eb[\norm{Y_i}_\infty^q]}{A^q}\\
    &\leq\frac{\Eb[Y_{ij}^4]}{A^2}+\frac{\Eb[\norm{Y_i}_\infty^q]}{A^{q-2}}.
\end{align*} We take $A=(\phi\log(d))^{-1}$, then \eqref{eq:pf_propA.1:2} with leads to
\begin{equation}\label{eq:pf_propA.1:4}
    \Delta_2(\phi)\lesssim \frac{(\phi\log(d))^2B_n^4}{n}+\frac{(\phi\log(d))^{q-2}D_n^q}{n^{q/2-1}}.
\end{equation} The quantity $\Delta_3(\phi)$ can be controlled as
\begin{align}\label{eq:pf_propA.1:5}
    \Delta_3(\phi)&\leq \sum_{i=1}^n \Eb\left[\norm{Y_i}_\infty^2\mathbf{1}\set{\norm{Y_i}_\infty>(\phi\log(d))^{-1}}\right]\nonumber\\
    &\leq \sum_{i=1}^n (\phi\log(d))^{q-2}\Eb\left[\norm{Y_i}_\infty^q\right]\lesssim\frac{(\phi\log(d))^{q-2}D_n^q}{n^{q/2-1}},
\end{align} where the last inequality follows from \eqref{eq:pf_propA.1:2}. Hence, combining \eqref{eq:pf_propA.1:3}, \eqref{eq:pf_propA.1:4}, \eqref{eq:pf_propA.1:5}, and \eqref{eq:pf_propA.1:1} yields
\begin{align*}
    \chi_n&\lesssim \frac{\phi\log^2(d)B_n}{\sqrt{n}}+\frac{\phi^3\log^{7/2}(d)B_n^2}{n}+\frac{\phi\log^2(d)B_n^2}{\sqrt{n}}+\frac{\phi^{q/2-1}\log^{q/2}(d)D_n^{q/2}}{n^{q/4-1/2}}\\
    &\qquad +\frac{\phi^3\log^{7/2}(d)B_n^4}{n}+\frac{\phi^{q-1}\log^{q-1/2}(d)D_n^{q}}{n^{q/2-1}}+\frac{\phi^{q/2-1}\log^{q/2+1/2}(d)D_n^{q/2}}{n^{q/4-1/2}}+\frac{\log^{1/2}(d)}{\phi},
\end{align*} for all $\phi>0$. Here, we used that $\sqrt{a+b}\leq\sqrt{a}+\sqrt{b}$. Since $B_n\geq1$ and $\log(d)\geq1$, we can further deduce that
\begin{equation*}
    \chi_n\lesssim\frac{\phi\log^2(d)B_n^2}{\sqrt{n}}+\frac{\phi^3\log^{7/2}(d)B_n^4}{n}+\frac{\phi^{q-1}\log^{q-1/2}(d)D_n^{q}}{n^{q/2-1}}+\frac{\phi^{q/2-1}\log^{q/2+1/2}(d)D_n^{q/2}}{n^{q/4-1/2}}+\frac{\log^{1/2}(d)}{\phi}.
\end{equation*} We choose
\begin{equation*}
    \phi^{-1}=\frac{B_n\log^{3/4}(d)}{n^{1/4}}+\frac{D_n\log(d)}{n^{1/2-1/q}}.
\end{equation*} This leads to
\begin{align*}
    \frac{\phi\log^2(d)B_n^2}{\sqrt{n}}+\frac{\phi^3\log^{7/2}(d)B_n^4}{n}&\leq\left(\frac{B_n\log^{3/4}(d)}{n^{1/4}}\right)^{-1}\frac{\log^2(d)B_n^2}{\sqrt{n}}+\left(\frac{B_n\log^{3/4}(d)}{n^{1/4}}\right)^{-3}\frac{\phi\log^{7/2}(d)B_n^4}{n}\\
    &\qquad=\frac{2B_n\log^{5/4}(d)}{n^{1/4}}.
\end{align*} Furthermore, 
\begin{align*}
    &\frac{\phi^{q-1}\log^{q-1/2}(d)D_n^{q}}{n^{q/2-1}}+\frac{\phi^{q/2-1}\log^{q/2+1/2}(d)D_n^{q/2}}{n^{q/4-1/2}}\\
    &\qquad\leq\left(\frac{D_n\log(d)}{n^{1/2-1/q}}\right)^{-(q-1)}\frac{\log^{q-1/2}(d)D_n^{q}}{n^{q/2-1}}+\left(\frac{D_n\log(d)}{n^{1/2-1/q}}\right)^{-(q/2-1)}\frac{\log^{q/2+1/2}(d)D_n^{q/2}}{n^{q/4-1/2}}\\
    &\qquad=\frac{D_n\log^{1/2}(d)}{n^{1/2-1/q}}+\frac{D_n\log^{3/2}(d)}{n^{1/2-1/q}}\leq\frac{D_n\log^{3/2}(d)}{n^{1/2-1/q}}.
\end{align*} Finally, one has
\begin{equation*}
    \frac{\log^{1/2}(d)}{\phi}=\frac{B_n\log^{5/4}(d)}{n^{1/4}}+\frac{D_n\log^{3/2}(d)}{n^{1/2-1/q}}.
\end{equation*} This proves the proposition.
\end{proof}

\bibliographystyle{apalike}
\bibliography{ref.bib}

\end{document}